%% file: nets.tex
\documentclass[11pt,a4paper]{amsart}
\usepackage{header}
\usepackage{tikz_def}

\title[Multi-Nets]{Multi-Nets. Classification of discrete and smooth surfaces with characteristic properties on arbitrary parameter rectangles}

\author[A.~I.~Bobenko, H.~Pottmann, and T.~R\"orig]{Alexander I. Bobenko \and Helmut Pottmann \and Thilo R\"orig}

\thanks{This research was supported by the DFG Collaborative Research Center TRR
109 “Discretization in Geometry and Dynamics” --
\url{www.discretization.de}.}

\address{Alexander I. Bobenko, Thilo R\"orig\vspace{-.5em}}
\address{Institute of Mathematics, Secr. MA 8-4, TU Berlin, 10623
Berlin, Germany\vspace{-.5em}}
\address{Email: {\normalfont \{bobenko,roerig\}@math.tu-berlin.de}}

\address{Helmut Pottmann\vspace{-.5em}}
\address{Geometric Modeling and Industrial Geometry, TU Wien, 
  1040 Vienna, Austria\vspace{-.5em}}
\address{Email: {\normalfont pottmann@geometrie.tuwien.ac.at}}


\begin{document}

\begin{abstract}
  We investigate the common underlying discrete structures for various smooth and discrete nets.
  The main idea is to impose the characteristic properties of the nets not only on elementary quadrilaterals but also on larger parameter rectangles.
  For discrete planar quadrilateral nets, circular nets, $Q^*$-nets and conical nets we obtain a characterization of the corresponding discrete \emph{multi}-nets.
  In the limit these discrete nets lead to some classical classes of smooth surfaces.
  Furthermore, we propose to use the characterized discrete nets as discrete extensions for the nets to obtain structure preserving subdivision schemes.
\end{abstract}

\maketitle

%\tableofcontents

%\newpage

\input{main}

\bibliographystyle{abbrv}
\bibliography{main}

\end{document}

%% file: main.tex
\section{Introduction}
\label{sec:intro}

The aim of discrete differential geometry is to find discrete analogs of notions and methods from classical smooth differential geometry.
Different discretizations of smooth parametrized surfaces are one particular instance of discrete differential geometry that have been studied from a purely mathematical as well as an applied perspective.
The different structure of smooth parameter lines is reflected in particular properties of the corresponding discrete quadrilateral nets.

In this article we derive a common framework that allows us to characterize surface patches by a purely discrete approach.
Our guiding idea is:
\begin{center}
  \emph{Characterize the smooth surface patches by the structure of the underlying discrete net.}
\end{center}
The idea of the characterizations is to study smooth and discrete nets that satisfy a characteristic property not only for elementary quadrilaterals, but for arbitrary parameter rectangles.
This leads to discrete nets which easily yield well known smooth surfaces classes in the limit:
discrete planar quadrilateral nets lead to projective translation surfaces (cf.~\cite{Degen:1994:PlaneSilhouettes1}),
circular nets lead to isothermic channel surfaces (cf.~\Cref{thm:multi-circular-smooth}), and
conical nets lead to surfaces with planar principal curvature lines (cf.~\Cref{thm:multi-conical-smooth}).

The unique benefits of the chosen perspective are:
\begin{itemize}
\item
  The point of view of discrete differential geometry provides clear view on the structure and leads to surface classes of classical smooth differential geometry.
\item
  The proofs of the Theorems are based on classical geometry and elementary observations on planar quadrilateral nets.
\item
  The chosen approach has applications to structure preserving/geometry respecting subdivision schemes and gives a unified perspective on smooth extensions of discrete nets, see \Cref{fig:teaser}.
\end{itemize}

In previous work we have studied different kinds of extensions of discrete nets by suitable surface patches.
This has led to the extensions of circular nets by Dupin cyclide patches~\cite{BobenkoHuhnenVenedey:Cyclides},
of discrete asymptotic nets by hyperboloid patches~\cite{HVRoer:HyperbolicNets}, and
of discrete conjugate nets by supercyclide patches~\cite{BobHVRoer:Supercyclides}.
All the surface patches used for the extensions, i.e., Dupin cyclide, hyperboloid, and supercyclide surface patches, are the ``simplest'' non-trivial surfaces contained in the corresponding smooth surface classes that are characterized by the multi-nets (cf.~\Cref{thm:Degen,thm:multi-circular-smooth,thm:multi-conical-smooth,thm:multi-congruence-lie,thm:multi-congruence-pluecker}).

From the Klein perspective on geometry, all these surfaces ``live'' in the corresponding subgeometries of projective geometry.
For most of the article we will use elementary representations of the nets in Euclidean 3-space.
In Section~\ref{sec:trans_in_quadric} we provide the characterization of nets in quadrics that can also be used with the appropriate models of M\"obius and Laguerre geometries to classify circular and conical nets, respectively.

\begin{figure}[t]
  \includegraphics[width=.45\textwidth]{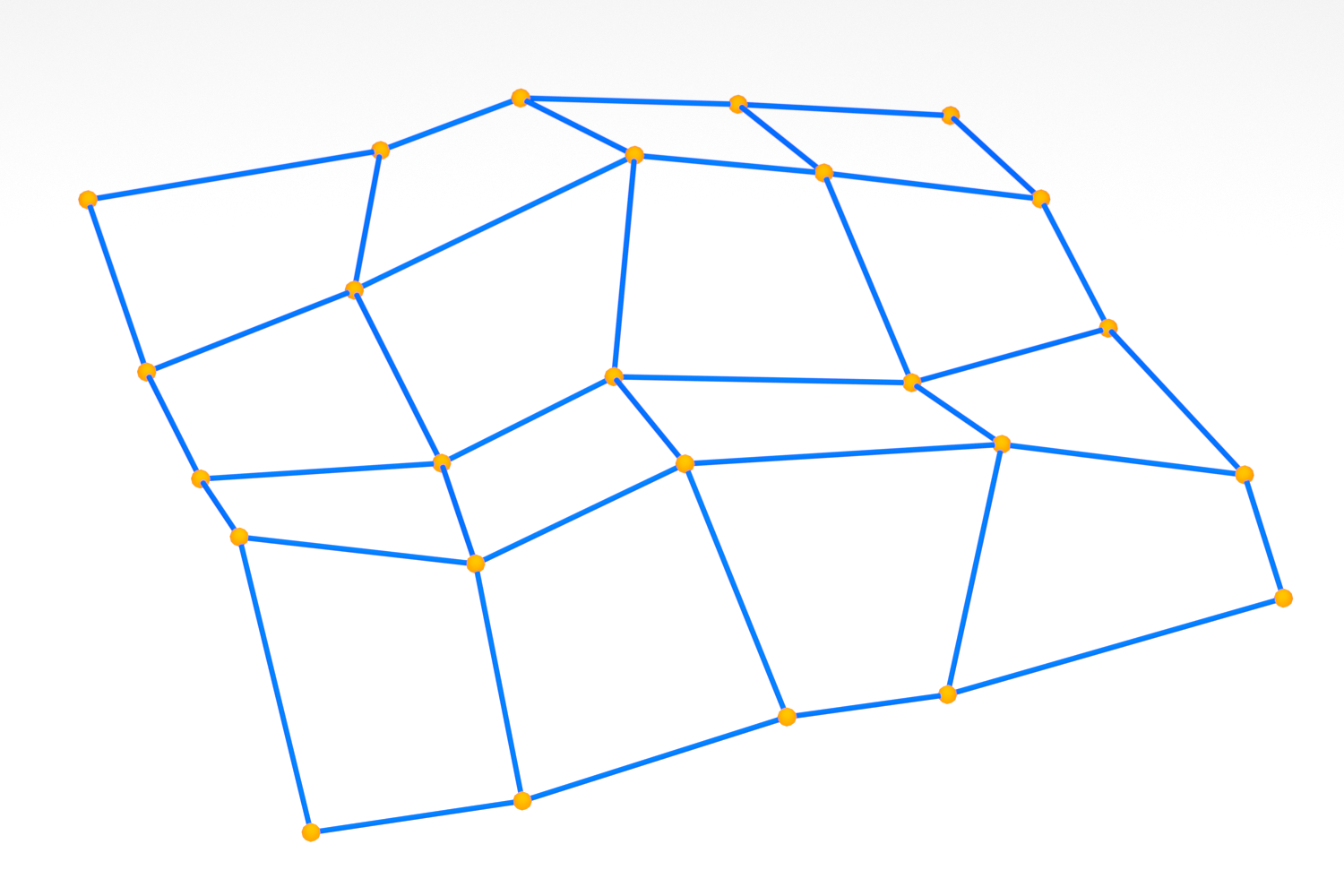}
  \includegraphics[width=.45\textwidth]{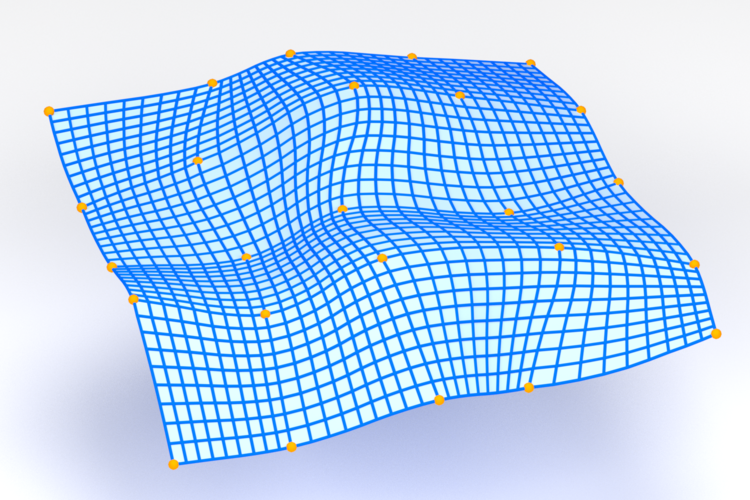}
  \caption{The discrete extension of a 4x4 planar quadrilateral net by a multi-planar net: The original net (left) and the subdivision by planar quadrilaterals (right).}
  \label{fig:teaser}
\end{figure}

\subsection*{Related work}

Discrete nets with planar quadrilateral faces have already been introduced by Sauer~\cite{Sauer:1937:ProjLinienGeometrie} as discrete analogs of surfaces parametrized along conjugate directions.
Modern investigations of these nets concentrate either on the integrable structure of these 
nets~\cite{BobenkoSuris:DDG, DoliwaSantini:1997:QnetsAreIntegrable} or on architectural 
applications~\cite{Mesnil:2016:Marionette, net:10.2312:CGF.v29i5pp1671-1679} and the design of 
nets with prescribed structure~\cite{Bouaziz:2012:ShapeUp, tang-2014-ff}.

For surfaces parametrized along principal curvature lines there exist two established discretizations in discrete differential geometry, namely circular and conical nets.
These nets have interesting properties from the purely mathematical point of view~\cite{BobenkoPottmannWallner:2010:CurvatureTheory, BobenkoSuris:DDG} but have also found various applications in the context of architectural geometry~\cite{BO20162, Bo:2011:CAS, Douthe:Isoradial, liu-2006-cm}.

In the smooth setting Degen~\cite{Degen:1994:PlaneSilhouettes1} studied parametrized surfaces with the property, that every parameter rectangle is planar. He showed that this property characterizes projective translation surfaces. If additionally, the dual condition holds, i.e., the tangent planes at the surface points of a parameter rectangle intersect in a point, then the surfaces are projective translation surfaces created from planar curves. In this and a subsequent paper~\cite{Degen:1997:PlaneSilhouettes2} he points out applications of these surfaces in Computer Aided Geometric Design.

Our work on discrete extensions of discrete nets is a contribution towards interpolatory subdivision. 
An interpolatory subdivision scheme for triangle meshes
which achieves $C^1$ smoothness in the limit for topologically regular meshes is the butterfly scheme by
Dyn, Gregory and Levin~\cite{Dynetal-1990-butterfly}. Zorin et al.~\cite{zorinetal-1996-intsubd} derived
an improved scheme which retains the simplicity of the butterfly scheme and results in smoother surfaces.
A simple interpolatory scheme for quadrilateral nets with arbitrary topology that generates 
$C^1$ surfaces in the limit, has been presented by Kobbelt~\cite{Kobbelt-1996-intsubd}. All these are
linear and local schemes which compute the positions of the new points using affine combinations of nearby
points from the unrefined net. There has also been research on nonlinear interpolatory subdivision.  There,
the main question has been smoothness analysis, which in most cases has been accomplished by
proximity to linear schemes \cite{grohs-2008-subd}. We are not aware of research on structure preserving
subdivision except for algorithms which combine subdivision with numerical optimization \cite{liu-2006-cm}. 
Very recently we became aware of research by Vaxman et al. \cite{vaxmanetal-2018-moebiussubd} on
interpolatory subdivision that commutes with M{\"o}bius transformations.

\subsection*{Notation and preliminaries}
\label{ssec:notations}

We use different fonts to distinguish points $p \in \RP^n$ or surfaces $x: [0,1]^2 \to \RP^n$ from the corresponding homogeneous coordinates $\hc{p} \in \R^{n+1}$ and $\hc{x} : [0,1]^2 \to \R^{n+1}$, respectively.
The values of discrete nets $x: \Z^2 \to \RP^n$ are denoted by lower indices, i.e., $x\ind{i,j} := x(i,j)$.
Upper indices are used for directions, e.g., $y^1\ind{i,j}$ for the Laplace point of the first direction in the discrete case or
$\laplace{u}f$ and $\laplace{v}f$ for the corresponding smooth Laplace transforms of a parametrized surface patch with parameters $u$ and $v$. 
The \emph{Laplace transforms} of a smooth surface~$f:[0,1]^2 \to \RP^3$ parametrized along conjugate directions are given by:
\[
  \laplace{u}f = f_u - b\, f, 
  \quad\text{and}\quad
  \laplace{v}f = f_v - a\, f,
\]
where $a, b, c$ satisfy the \emph{Laplace equation}:
\[
  f_{uv} = a\, f_u + b\, f_v + c f.
\]
The Laplace transforms of a surface are also known as the focal surfaces of the two tangent congruences.
If the Laplace transform~$\laplace{u}f$ degenerates to a curve, then all the $u$-tangents along a $v$-curve intersect in a point.
Hence the $v$-curves are silhouette curves of the surface (imagine a point light on the Laplace transform curve).

The rest of the article is organized as follows. We start with planar quadrilateral nets and their duals in~\Cref{sec:qnets,sec:qstar,sec:qqstar}.
Then we investigate the structure of smooth extensions using different kinds of projective translation surface patches in~\Cref{sec:piecewise_smooth_extension}.
In~\Cref{sec:trans_in_quadric} we describe multi-nets in quadrics, which can be used to investigate circular and conical nets and their generation in the sense of F.~Klein's Erlangen Program.
The discretizations of principal curvature lines as circular and conical nets are treated in~\Cref{sec:circularnets,sec:conicalnets}.
In~\Cref{sec:discrete-extensions} we propose the discrete extension of discrete nets, that will lead to the smooth extensions in the smooth limit.
Finally, in~\Cref{sec:isotropic_line_congruences} we extend the investigations from nets to line congruences and obtain characterizations of multi-line congruences in Lie and Pl\"ucker geometries.

\input{qnets}

\input{qstarnets}

\input{qqstarnets}

\input{smoothextensions}

\input{qnets-quadrics}

\input{circularnets}

\input{conicalnets}

\input{discrete-extensions}

\input{linecongruences}

\subsubsection*{Acknowledgements}
We would like to thank Wolfgang Schief and Jan Techter for many fruitful discussions.
This research was supported by the DFG Collaborative Research Center TRR 109, ``Discretization in Geometry and Dynamics.'' Helmut Pottmann was supported through Grant I 2978-N35 of the Austrian Science Fund (FWF).

%%% Local Variables: 
%%% mode: latex
%%% TeX-master: "nets"
%%% End: 

%% file: qnets.tex
\section{Q-nets}
\label{sec:qnets}

In this section we introduce discrete Q- and multi-Q-nets and characterize multi-Q-nets in $\RP^n$.
In \Cref{thm:multiqnets} we prove equivalent characterizations of multi-Q-nets.
In particular we show that multi-Q-nets are discrete projective translation surfaces.

Surface parametrizations along conjugate directions have the property that infinitesimal parameter rectangles are planar.
This leads to the following natural definition of discrete conjugate nets.

\begin{definition}
  \label{def:qnet}
  A 2-dimensional \emph{discrete conjugate net} or \emph{Q-net} is a
  map $x: \Z^2 \to \RP^n$ such that the points $x\ind{i,j}, x\ind{i+1,j},
  x\ind{i+1,j+1}, x\ind{i,j+1}$ are coplanar.
\end{definition}

\begin{remark}
  We will always assume that all data is in general position, without specifying this explicitly in our statements and reasonings. In particular, in the following theorem it will be assumed, that the strips of the Q-net are not contained in a plane.
\end{remark}

\begin{figure}[t]
  \centering
  \def\tikzscale{.5} \input{tikz/laplace_points.tikz}
  \caption{Notation for Laplace points of an elementary quadrilateral}
  \label{fig:laplace_points}
\end{figure}
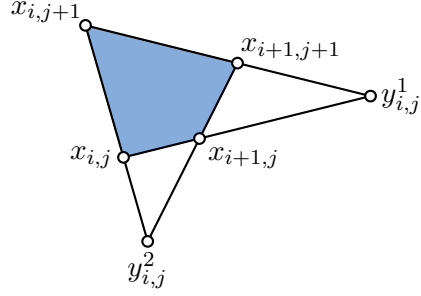

The vertices of a generic elementary quadrilateral of a Q-net satisfy the Laplace equation:
\[
  \hc x\ind{i+1,j+1} =
  a\, \hc x\ind{i+1,j} + b\, \hc x\ind{i,j+1} - c\, \hc x\ind{i,j}
\]
for some $a,b,c \neq 0$.
The coefficients $a,b,c$ may be used to obtain normalized homogeneous coordinates~$\tilde{\hc{x}}$ such that 
\[
  \tilde{\hc{x}}\ind{i+1,j+1} =
  \tilde{\hc{x}}\ind{i+1,j}
  + \tilde{\hc{x}}\ind{i,j+1}
  - \tilde{\hc{x}}\ind{i,j}
\]

Using the Laplace equations above, we obtain the intersection points of opposite pairs of edges
\begin{align*}
  \hc y^1\ind{i,j}
  &= \tilde{\hc{x}}\ind{i+1,j+1} - \tilde{\hc{x}}\ind{i,j+1}
  = \tilde{\hc{x}}\ind{i+1,j} - \tilde{\hc{x}}\ind{i,j} \quad\text{and}\\
  \hc y^2\ind{i,j}
  &= \tilde{\hc{x}}\ind{i+1,j+1} - \tilde{\hc{x}}\ind{i+1,j}
  = \tilde{\hc{x}}\ind{i,j+1} - \tilde{\hc{x}}\ind{i,j}.
\end{align*}
The points $y^1 = [\hc y^1]$ and $y^2 = [\hc{y}^2]$ are called \emph{Laplace points} (see \Cref{fig:laplace_points}), where the superscript denotes the direction of the opposite pair of edges.
The nets $y^1\ind{i,j}$ and $y^2\ind{i,j}$ generated by the Laplace points are called the \emph{Laplace transforms} of the Q-net. 

We will now define and investigate discrete nets with planar parameter quadrilaterals similar to the nets investigated in~\cite{Degen:1994:PlaneSilhouettes1} in the smooth case. 

\begin{definition}
  \label{def:multiqnet}
  A 2-dimensional \emph{multi-Q-net} is a map $x: \Z^2 \to \RP^n$ such that for every $i_0 \neq i_1$ and $j_0 \neq j_1$ the coordinate quadrilateral $x\ind{i_0, j_0}, x\ind{i_0, j_1},$ $x\ind{i_1, j_1}, x\ind{i_1, j_0}$ is planar.
\end{definition}

An example of such a net is shown in \Cref{fig:multiqnet} (left), which is a Euclidean translation net generated from two planar polygons.

\begin{figure}[tb]
  \centering
  \includegraphics[width=.45\textwidth]{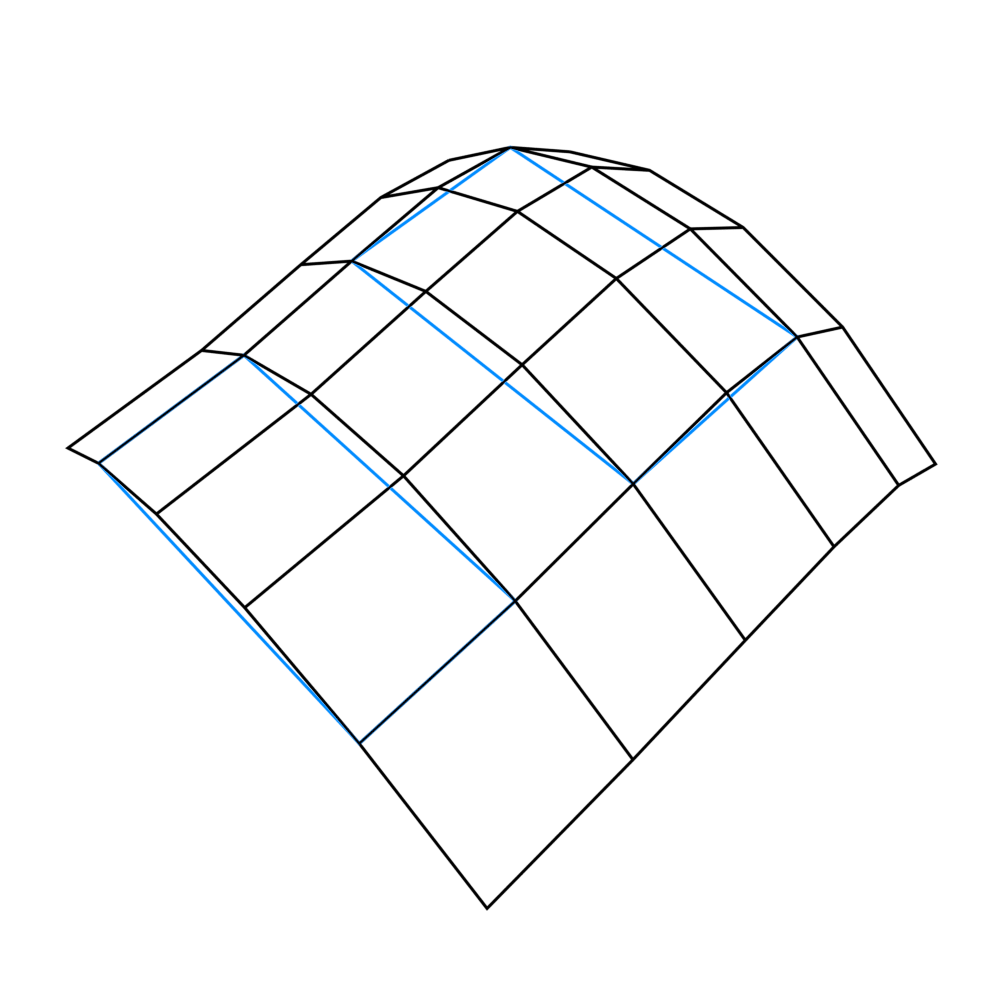}
  \hfill
  \includegraphics[width=.45\textwidth]{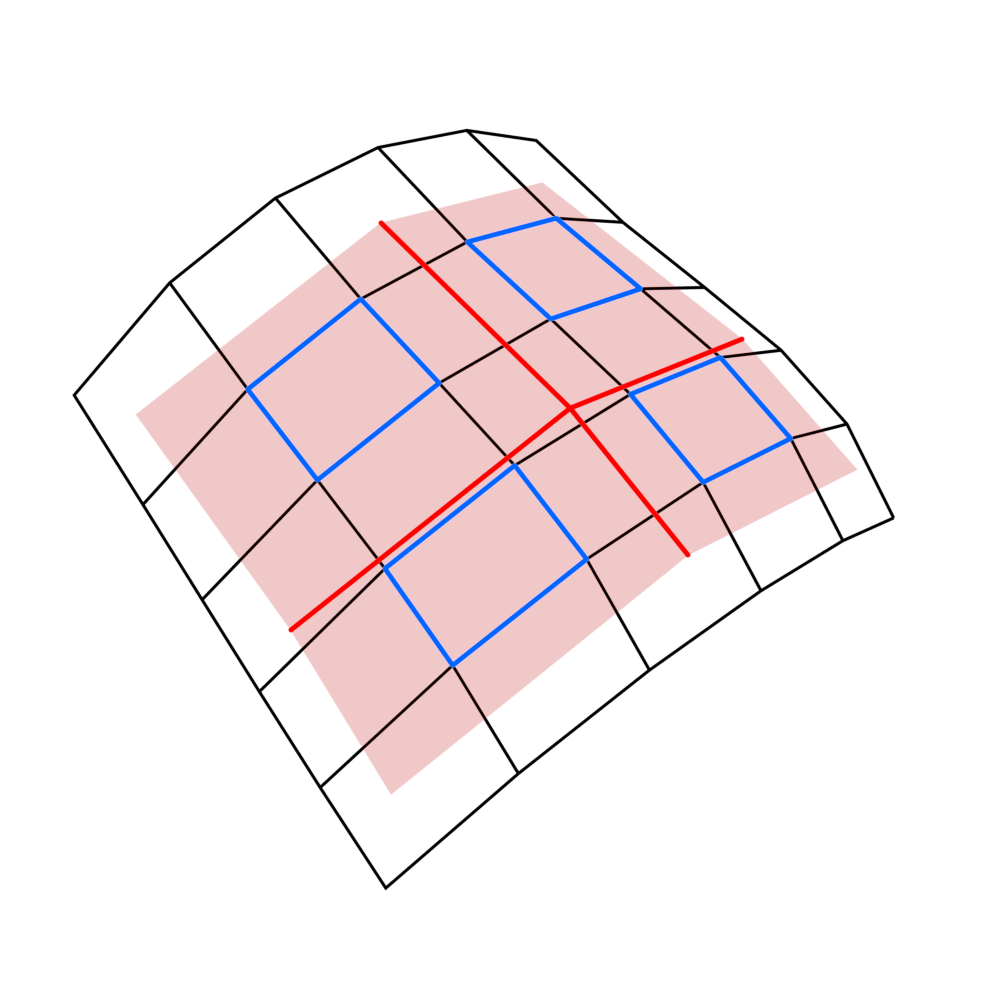}
  \caption{
    Left: A discrete multi-conjugate net as defined in \Cref{def:multiqnet}. All parameter quadrilaterals are planar. 
    Right: A discrete multi-$Q^*$-net as defined in \Cref{def:multiqstar}.
    The four red planes spanned by the quadrilaterals with blue boundary edges of arbitrary parameter rectangles intersect in a point.
  }
    \label{fig:multiqnet}
\end{figure}

\begin{theorem}
  \label{thm:multiqnets}
  Let $x :\Z^2 \to \RP^n$ be a Q-net in~$\RP^n$.
  Then the following are equivalent
  \begin{enumerate}[(i)]
  \item\label{item:multi-q-net} $x$ is a multi-Q-net.
  \item\label{item:all-perspective} Every two parameter lines of~$x$ of the same direction are in perspective with respect to a point.
  \item\label{item:neighbor-perspective} Every two neighboring discrete parameter lines of~$x$ are in perspective with respect to a point.
  \item\label{item:translation}
    There exist two discrete curves
    $\hc{p}: \Z \to \R^{n+1}$ and $\hc{q}: \Z \to \R^{n+1}$, and a point~$\hc{x}\ind{00} \in \R^{n+1}$ such that
    $
    x\ind{i,j} = [\hc{p}\ind{i} + \hc{q}\ind{j}].
    $
  \item\label{item:degenerate-laplace}
    $x$ is a $Q$-net whose two Laplace transforms degenerate to curves.
  \end{enumerate}
\end{theorem}

\begin{proof}
  \eqref{item:multi-q-net} $\Rightarrow$ \eqref{item:all-perspective}:
  So assume that $x$ is a multi-Q-net.
  Then generically, the points of two neighboring quadrilaterals span a 3-dimensional subspace. Consider the lines $\ell^1_j = \lspan{x\ind{i,j}, x\ind{i+1,j}}$ spanned by the edges in first net direction as shown in \Cref{fig:perspective_multiqnet} (left).
  Since the quadrilaterals are planar, the lines $\ell^1_j$ and $\ell^1_{j+1}$ intersect.
  Additionally, the lines $\ell^1_{j}$ and $\ell^1_{j+2}$ intersect, since $x$ is a multi-conjugate net.
  So every two of the three lines $\ell^1_{j}$, $\ell^1_{j+1}$, and $\ell^1_{j+2}$ intersect.
  Since the three lines do not lie in one plane by genericity assumption, the three lines intersect in one point. 
  This is the common Laplace point $y^1\ind{i,j} = y^1\ind{i,j+1}$.
  Hence all Laplace points of a (generic) strip coincide and the parameter lines $x(i,\cdot)$ and $x(i+1,\cdot)$ are in perspective with respect to the Laplace point $y^1\ind{i} :=   y^1\ind{i,j}$ of the coordinate strip as shown in \Cref{fig:perspective_multiqnet}.

  \eqref{item:all-perspective} $\Rightarrow$ \eqref{item:neighbor-perspective}:
  Obvious.
    
  \eqref{item:neighbor-perspective} $\Rightarrow$ \eqref{item:multi-q-net}:
  Assume that every two neighboring parameter polygons are in perspective with respect to a point.
  Then the lines spanned by the edges across the coordinate strips intersect in the center of the perspectivity (see \Cref{fig:perspective_multiqnet}, left).
  So all parameter rectangles
  $(x\ind{i,j}, x\ind{i,j+k}, x\ind{i+1,j+k}, x\ind{i+1,j})$ and
  $(x\ind{i,j}, x\ind{i+k,j}, x\ind{i+k,j+1}, x\ind{i,j+1})$ along coordinate strips are planar.
  In particular, $x$ is a Q-net.
  
  Now consider a $2\times 2$ piece of the net as shown in \Cref{fig:perspective_multiqnet}.
  We need to show that the lines
  $g^{1}_j = \lspan{x\ind{i,j},x\ind{i+2,j}}$,
  $g^1_{j+1} = \lspan{x\ind{i,j+1},x\ind{i+2,j+1}}$,
  $g^1_{j+2} = \lspan{x\ind{i,j+2},x\ind{i+2,j+2}}$
  intersect in a point.
  The line~$h$ connecting the two Laplace points $y^1_{i,j}$ and $y^1_{i+1,j}$ is contained in the three planes spanned by
  $(x\ind{i,j}, x\ind{i+1,j}, x\ind{i+2,j})$,
  $(x\ind{i,j+1}, x\ind{i+1,j+1}, x\ind{i+2,j+1})$, and
  $(x\ind{i,j+2}, x\ind{i+1,j+2}, x\ind{i+2,j+2})$, respectively.
  Since the strip rectangles are planar, the lines $g^{1}_j,g^1_{j+1}$ and $g^1_{j+1},g^1_{j+2}$ intersect.
  The intersection points
  $p_j = g^{1}_j \cap g^1_{j+1} $ and
  $p_{j+1} =  g^1_{j+1} \cap g^1_{j+2}$ have to lie on the line~$h$ connecting the two Laplace points.
  Hence $p_j = g_{j+1} \cap h = p_{j+1}$ and the quadrilateral
  $(x\ind{i,j},x\ind{i+2,j},x\ind{i+2,j+2},x\ind{i,j+2})$ is planar.
  This implies the planarity of abitrary parameter rectangles and thus $x$ is a multi-Q-net.

  \eqref{item:multi-q-net}~$\Rightarrow$~\eqref{item:translation}:
  Let $x$ be a multi-Q-net.
  Then for an elementary quadrilateral of~$x$ there exist homogeneous coordinates~$\hc x$ for the vertices and the two Laplace points~$y^1, y^2$ such that
  \begin{align*}
    \hc{y}^1\ind{i,j} &=  \hc{x}\ind{i+1,j}-\hc{x}\ind{i,j} = \hc{x}\ind{i+1,j+1}-\hc{x}\ind{i,j+1} 
    \quad\text{and} \\
    \hc{y}^2\ind{i,j} &= \hc{x}\ind{i,j+1}-\hc{x}\ind{i,j} = \hc{x}\ind{i+1,j+1}-\hc{x}\ind{i+1,j}\;.
  \end{align*}
  Hence $\hc{x}\ind{i+1,j+1} =
  \hc{x}\ind{i,j}+\hc{y}^1\ind{i,j}+\hc{y}^2\ind{i,j}$. Since the Laplace points
  are unique for entire coordinate strips in the net, we find coordinates such that
  \begin{align*}
    \hc{x}\ind{i+1,j} &= \hc{x}\ind{i,j} + \hc{y}^1\ind{i} 
    \quad \text{for all $j$, and} \\
    \hc{x}\ind{i,j+1} &= \hc{x}\ind{i,j} + \hc{y}^2\ind{j} 
    \quad \text{for all $i$.}
  \end{align*}
  Thus
  \begin{align}
  \hc{x}\ind{i,j} = \hc{x}\ind{0,0} 
  + \sum_{k=0}^{i-1}\hc{y}^1\ind{k}
    + \sum_{\ell=0}^{j-1}\hc{y}^2\ind{\ell}\;.
    \label{eq:translation_laplace_curves}
  \end{align}
  With 
  $\hc{p}\ind{i} = \tfrac12 \hc{x}\ind{0,0} + \sum_{k=0}^{i-1}\hc{y}^1\ind{k}$, and 
  $\hc{q}\ind{j} = \tfrac12 \hc{x}\ind{0,0} + \sum_{\ell=0}^{j-1}\hc{y}^2\ind{\ell}$ 
  we obtain the desired representation of the net~$x$ for non-negative~$i$ and~$j$.
  This proves the desired structure for the positive orthant.
  By the same arguments, the polygons~$\hc{p}$ and~$\hc{q}$ can easily be extended to the other three orthants.
  
  \eqref{item:translation} $\Rightarrow$ \eqref{item:degenerate-laplace}:
  If $x$ is a projective translation surface, then there exist $\hc p_i$ and $\hc q_j$, such that  $\hc x = \hc p_i + \hc q_j$. But then Laplace transforms are given by
  \begin{align*}
    \hc y^1\ind{i,j}
    &= \hc x\ind{i+1,j} - \hc x\ind{i,j}
    = \hc p_i \quad\text{independent of $j$, and}\\
    \hc y^2\ind{i,j}
    &= \hc x\ind{i,j+1} - \hc x\ind{i,j}
    = \hc q_j \quad\text{independent of $i$.}
  \end{align*}
  So the Laplace transforms degenerate to the curves~$y^1\ind{i} := [\hc p\ind{i}]$ and~$y^2\ind{j} = [\hc q_j]$.

  \eqref{item:degenerate-laplace} $\Rightarrow$ \eqref{item:neighbor-perspective}:
  Let $x$ be a $Q$-net with degenerate Laplace transforms~$y^1\ind{i}$ and $y^2\ind{j}$.
  Then all the lines spanned by the edges $(x\ind{i,j},x\ind{i+1,j})$ (resp.\ $(x\ind{i,j}, x\ind{i,j+1})$) pass through the common Laplace point~$y^1\ind{i}$ (resp.~$y^2\ind{j}$).
  Thus neighboring parameter lines are in perspective with respect to the corresponding Laplace points.
\end{proof}

\begin{figure}[t]
  \centering
%  \includetikz{0.6}{laplace_points_2quads.tikz}
  \def\tikzscale{0.8} \input{tikz/2x2-quads.tikz}
  \caption{The Laplace points of a multi-Q-net coincide along
    quadrilateral strips. Hence arbitrary parameter lines of
    multi-Q-nets are in perspective with respect to common
    Laplace points.}
  \label{fig:perspective_multiqnet}
\end{figure}
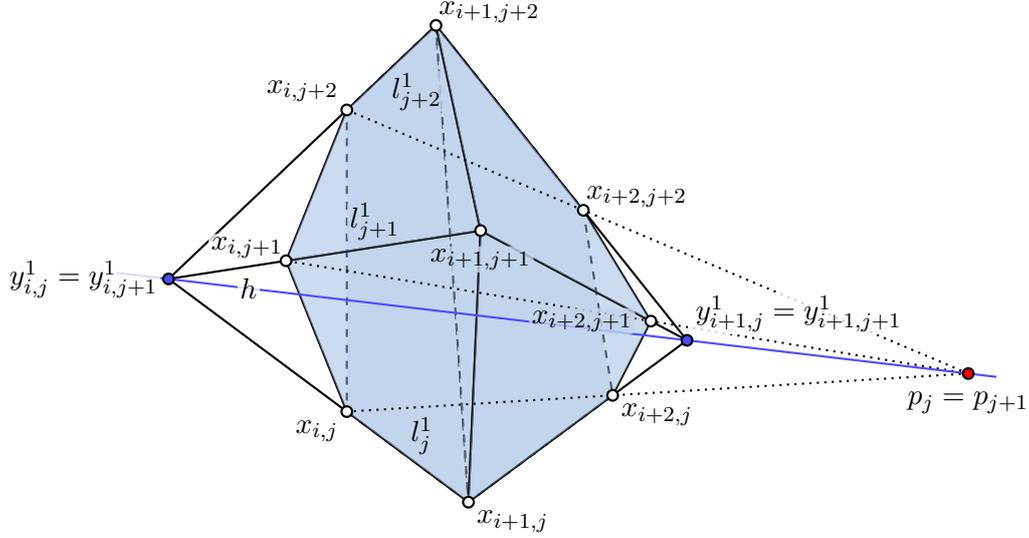

Geometrically, this means, that for multi-Q-nets, there exist homogeneous coordinates for the vertices in~$\R^{n+1}$, such that all the faces are parallelograms.

If all parameter quadrilaterals of a smooth surface are planar, then two parameter curves of one of the two families are in perspective with respect to a point.
In the limit, this implies, that the Laplace transforms degenerate to curves.
These surfaces with degenerate Laplace transforms have been characterized and they turn out to be projective translation surfaces~(see~\cite{Bol:ProjektiveDifferentialgeometrie}, or \cite[Sect.~2]{Degen:1994:PlaneSilhouettes1} for a proof).
A parameter line of one direction with the property, that all tangents of the other direction intersect in one point, i.e., lie on a cone, is called a silhouette line.

\paragraph{Cauchy problem}
As a consequence of the above theorem we can read off the initial data needed for the construction of a multi-$Q$-net.

\emph{Homogeneous initial data.}
Given the two degenerate Laplace transform curves $\hc{y}^1, \hc{y}^2: \Z \to \R^4$ and an initial point $\hc{x}\ind{00} \in \R^4$.
Then there exists a unique multi-Q-net~$x:\Z^2 \to \RP^3$ given by equation~\eqref{eq:translation_laplace_curves}.

\emph{Projective initial data.}
Given the two Laplace transformed curves~$y^1, y^2\in \RP^3$ and an initial point~$x\ind{00} \in \RP^3$.
Then this will \emph{not} define a unique multi-Q-net, since the combination of Laplace transform and initial point fixes a line, but not a point on this line.
So we need an extra scalar value (mimicking the homogeneous coordinate) for the edges of the coordinate axes $x\ind{\cdot,j}$ and $x\ind{i,\cdot}$ to obtain a unique multi-Q-net. 

To avoid this extra scalar function on the coordinate axes, we can prescribe two quadrilateral coordinate strips satisfying the perspectivity property of Theorem~\ref{thm:multiqnets}~\eqref{item:neighbor-perspective} for the construction of multi-Q-nets in~$\RP^3$, see \Cref{fig:projective_translation_cauchy}.

\begin{figure}[bt]
  \includegraphics[width=.5\linewidth]{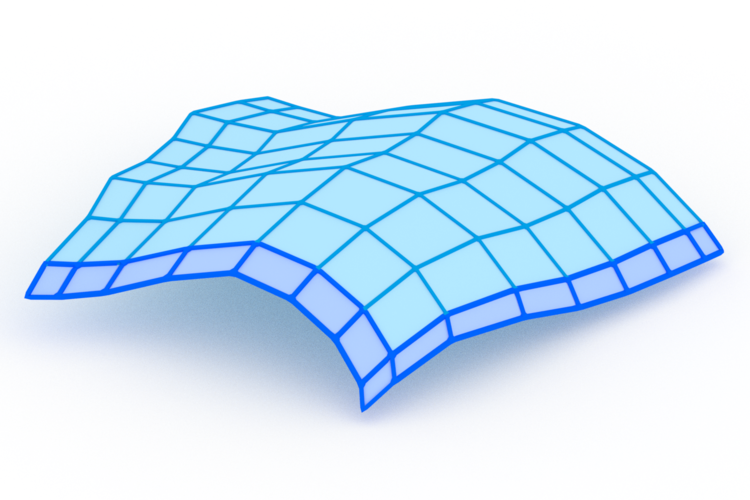}
  \caption{Discrete multi-Q-net constructed from two initial quadrilateral strips.}
  \label{fig:projective_translation_cauchy}
\end{figure}

%%% Local Variables: 
%%% mode: latex
%%% TeX-master: "nets"
%%% End: 

%% file: tikz/laplace_points.tikz
\begin{tikzpicture}[scale=\tikzscale]

\coordinate (xij1) at (-3,3);
\coordinate (xi1j) at (0,0);
\coordinate (xi1j1) at (1,2);
\coordinate (xij) at (-2,-0.5);

\coordinate (y1) at (intersection of xij--xi1j and xij1--xi1j1);
\coordinate (y2) at (intersection of xij--xij1 and xi1j--xi1j1);

\draw[fill=dgdblue!60] (xij)--(xi1j)--(xi1j1)--(xij1)--cycle;

\draw[myline] (xij)--(y1);
\draw[myline] (xij1)--(y1);
\draw[myline] (xij1)--(y2);
\draw[myline] (xi1j1)--(y2);

\node[mylabel] at (xij) [left=.1] {$x\ind{i,j}$};
\node[mylabel] at (xi1j) [below right=.1] {$x\ind{i+1,j}$};
\node[mylabel] at (xij1) [above left] {$x\ind{i,j+1}$};
\node[mylabel] at (xi1j1) [above right] {$x\ind{i+1,j+1}$};
\node[mylabel] at (y1) [right=.1] {$y^1\ind{i,j}$};
\node[mylabel] at (y2) [below=.1] {$y^2\ind{i,j}$};

\foreach \p in {xij, xij1, xi1j, xi1j1, y1, y2}
	\path (\p) [mypoint=white];

\end{tikzpicture}

%% file: tikz/2x2-quads.tikz
\begin{tikzpicture}[scale=\tikzscale]

\coordinate (xij) at (-2,0);
\coordinate (xi1j) at (0,-1.5);
\coordinate (xi1j1) at (.2,3);
\coordinate (xij1) at (-3,2.5);
\coordinate (xij2) at (-2,5);

\coordinate (xi2j1) at (3,1.5);

\coordinate (y1) at (intersection of xij--xi1j and xij1--xi1j1);
\coordinate (y2) at (intersection of xij--xij1 and xi1j--xi1j1);
\coordinate (xi2j) at ($(y2)!.8!(xi2j1)$);

\coordinate (xi1j2) at ($(y1)!1.5!(xij2)$);

\coordinate (y11) at (intersection of xi1j--xi2j and xi1j1--xi2j1);
\coordinate (y22) at (intersection of xij1--xij2 and xi1j1--xi1j2);

\coordinate (xi2j2) at (intersection of y11--xi1j2 and y22--xi2j1);

\coordinate (p) at (intersection of xij2--xi2j2 and xij--xi2j);

\draw[myline, fill=lightblue!20, dashed] (xij)--(xi1j)--(xi1j2)--(xij2)--cycle;
\draw[myline, fill=lightblue!20, dashed] (xi2j)--(xi1j)--(xi1j2)--(xi2j2)--cycle;
\draw[myline] (xij)--(xi1j)--(xi1j1)--(xij1)--cycle;
\fill[lightblue, opacity=.5] (xij)--(xi1j)--(xi1j1)--(xij1)--cycle;
\draw[myline] (xij2)--(xi1j2)--(xi1j1)--(xij1)--cycle;
\fill[lightblue, opacity=.5] (xij2)--(xi1j2)--(xi1j1)--(xij1)--cycle;

\draw[myline] (xi2j)--(xi1j)--(xi1j1)--(xi2j1)--cycle;
\fill[lightblue, opacity=.5] (xi2j)--(xi1j)--(xi1j1)--(xi2j1)--cycle;
\draw[myline] (xi2j2)--(xi1j2)--(xi1j1)--(xi2j1)--cycle;
\fill[lightblue, opacity=.5] (xi2j2)--(xi1j2)--(xi1j1)--(xi2j1)--cycle;

\draw[myline] (xij)--(y1);
\draw[myline] (xij1)--(y1);
\draw[myline] (xij2)--(y1);

\draw[myline] (xi2j)--(y11);
\draw[myline] (xi2j1)--(y11);
\draw[myline] (xi2j2)--(y11);

\draw[myline, dotted] (p)--(xij);
\draw[myline, dotted] (p)--(xij1);
\draw[myline, dotted] (p)--(xij2);
\draw[myline, blue!70] ($(p)!-.1!(y11)$)--($(y1)!-.1!(y11)$);
\node[mylabel] at ($(p)!.9!(y1)$) {$h$};

\node[mylabel] at (xij) [below left=.1] {$x\ind{i,j}$};
\node[mylabel] at (xi1j) [below right=.1] {$x\ind{i+1,j}$};
\node[mylabel] at (xi2j) [below right=.1] {$x\ind{i+2,j}$};

\node[mylabel] at (xij1) [above left=.0] {$x\ind{i,j+1}$};
\node[mylabel, fill=lightblue!60] at (xi1j1) [below=.2] {$x\ind{i+1,j+1}$};
\node[mylabel, fill=lightblue!60] at (xi2j1) [left=.2] {$x\ind{i+2,j+1}$};

\node[mylabel] at (xij2) [above left=.1] {$x\ind{i,j+2}$};
\node[mylabel] at (xi1j2) [above right] {$x\ind{i+1,j+2}$};
\node[mylabel] at (xi2j2) [above right] {$x\ind{i+2,j+2}$};

\node[mylabel] at (y1) [left=.1] {$y^1\ind{i,j}=y^1\ind{i,j+1}$};
\node[mylabel] at (y11) [above right=.1] {$y^1\ind{i+1,j}=y^1\ind{i+1,j+1}$};

\node[mylabel] at (p) [below=.2] {$p_j = p_{j+1}$};

\node[mylabel, fill=lightblue!60] at ($(xij2)!0.47!(xi1j2)$) [below right] {$l^1_{j+2}$}; 
\node[mylabel, fill=lightblue!60] at ($(xij1)!0.6!(xi1j1)$) [above left] {$l^1_{j+1}$}; 
\node[mylabel, fill=lightblue!60] at ($(xij)!0.5!(xi1j)$) [above right] {$l^1_{j}$}; 

\foreach \p in {xij, xij1, xi1j, xi1j1, xij2, xi1j2, xi2j, xi2j1, xi2j2}
	\path (\p) [mypoint=white];

\foreach \p in {y1, y11}
	\path (\p) [mypoint=blue!70];

\path (p) [mypoint=red];

\end{tikzpicture}

%% file: qstarnets.tex
\section{$Q^*$-nets}
\label{sec:qstar}

Looking at elementary quadrilaterals, 2-dimensional $Q$- and $Q^*$-nets in $\RP^3$ are just dual descriptions of similar objects, i.e., every $Q$-net defines a $Q^*$-net and vice versa.
Hence the statements proved in the previous section may easily be translated to $Q^*$-nets.

\begin{definition}
  \label{def:qstar}
  A 2-dimensional \emph{$Q^*$-net} is a map $P: \Z^2 \to
  \{\text{planes in }\RP^3\}$ such that the four planes $P\ind{i,j},
  P\ind{i+1,j}, P\ind{i+1,j+1}, P\ind{i,j+1}$ intersect.
\end{definition}

Once the definition is extended to parameter rectangles, $Q$- and $Q^*$-nets become distinct and reveal different properties. 

\begin{definition}
  \label{def:multiqstar}
  A $2$-dimensional \emph{multi-$Q^*$-net} is a map $P: \Z^2 \to \{\text{planes in }\RP^3\}$ such that the four planes $P\ind{i_0, j_0}, P\ind{i_1,j_0},  P\ind{i_1,j_1}, P\ind{i_0,j_1}$ intersect.
\end{definition}

Since we are considering nets in $\RP^3$, points and planes are dual to each other and lines are dual to lines, e.g., lines generated by two points are dual to planes intersecting in a line. For neighboring quadrilaterals of a $Q^*$-net we obtain the following statement dual to \Cref{thm:multiqnets}.
We omit some of the equivalent statements, since their dual formulation is less intuitive and we will not use them in the following.

\begin{theorem}
  \label{thm:coplanar_edges}
  Let $P:\Z^2 \to \{ \text{planes in } \RP^3 \}$ be a generic~$Q^*$-net. Then the following are equivalent:
  \begin{enumerate}[(i)]
  \item\label{item:multiqstar} $P$ is a multi-$Q^*$-net.
  \item\label{item:planarpolygons} $P$ has planar parameter lines.
  \end{enumerate}
\end{theorem}

\begin{proof}
  \eqref{item:multiqstar} is dual to \Cref{thm:multiqnets}~\eqref{item:multi-q-net} and \eqref{item:planarpolygons} is dual to \Cref{thm:multiqnets}~\eqref{item:all-perspective}. Hence the proof follows by duality in~$\RP^3$.
\end{proof}

Smooth surfaces parameterized along conjugate directions have the property that the tangent planes of an infinitesimal parameter rectangle intersect in a point.
Furthermore, a smooth surface parametrized by conjugate directions such that the tangent planes of arbitrary parameter rectangles intersect has planar parameter lines.
This follows easily from taking appropriate limits of the discrete mulit-$Q^*$-nets.
A detailed description of the smooth case is given in~\cite{Degen:1994:PlaneSilhouettes1}.

%%% Local Variables:
%%% mode: latex
%%% TeX-master: "nets"
%%% End:

%% file: qqstarnets.tex
\section{$(Q + Q^*)$-nets}
\label{sec:qqstar}

In the smooth setting, parametrizations along conjugate directions satisfy the $Q$-net and the $Q^*$-net property infinitesimally.
In the discrete case, this leads to a family of nets classified in the following theorem.

\begin{theorem}
  \label{thm:qqstarnets}
  Let $x: \Z^2 \to \RP^3$ be a multi-$Q$- and multi-$Q^*$-net. Then in the generic case
  \begin{compactenum}
  \item all Laplace points lie on a line, and
  \item all parameter line planes belong to a pencil.
  \end{compactenum}
\end{theorem}
\begin{proof}
  Let $y^1\ind{i}$ and $y^2\ind{j}$ denote the Laplace points of the coordinate strips, and $Y^1\ind{j}$ and $Y^2\ind{i}$ denote the planes containing the parameter polygons $x(\cdot,j)$ and $x(i,\cdot)$, respectively. 
  
  Then all the planes $Y^1\ind{j}$ (resp.~$Y^2\ind{i}$) pass through the Laplace points $y^1\ind{i}$ (resp.~$y^2\ind{j}$), since
  \[
    y^1\ind{i} \in \lspan{x\ind{i,j},x\ind{i+1,j}} \subset Y^1\ind{j}.
  \]
  If we assume, that neither all Laplace points of one direction nor
  all parameter line planes of one direction coincide, we obtain the
  desired configuration of Laplace points and coordinate planes.

  If, for example, all the coordinate lines of one direction lie in
  one plane, e.g. $Y^1\ind{j} = \mathtt{const}$, then the entire net is
  planar. If, dually, all Laplace points of one direction coincide,
  then the net is a cone whose unique Laplace point of this direction is the apex.
\end{proof}

These nets can be constructed from two initial planar polygons and two linear families of Laplace points on the respective planes.
An example of such a net is shown in \Cref{fig:qqstar-net}.
All the points of the example net lie on a supercyclide, as the two Laplace transforms of an arbitrary supercyclide are contained in two lines.

\begin{figure}[tb]
  \includegraphics[width=.7\textwidth, angle=180 ]{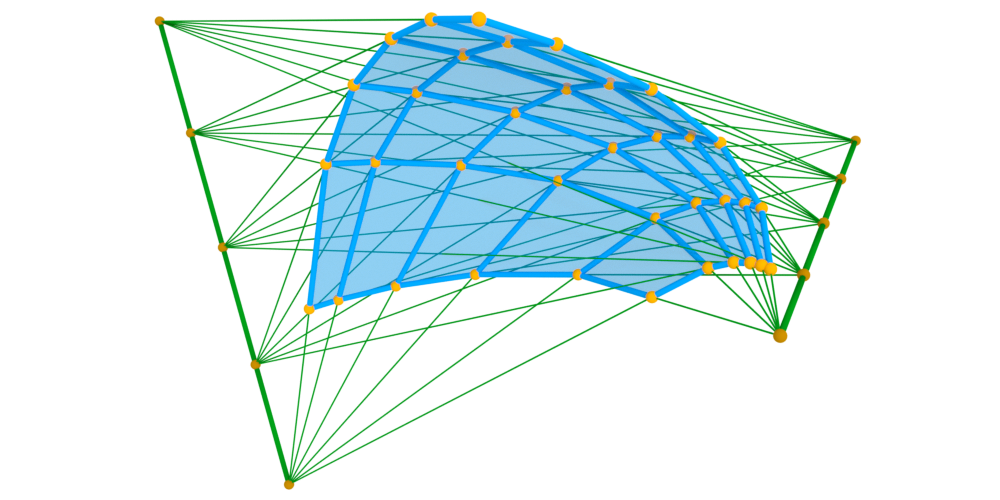}
  \caption{A $(Q+Q^*)$-net with its two Laplace transforms contained in skew lines.}
  \label{fig:qqstar-net}
\end{figure}

Smooth surfaces whose Laplace transforms are degenerate curves contained in a line are called $D$-surfaces and were studied in~\cite{Degen:1994:PlaneSilhouettes1,Degen:1997:PlaneSilhouettes2}.
In particular, parametrized surface patches $x: G \subset \R^2 \to \RP^3$ with the following two properties:
\begin{itemize}
\item[(A)]
  For every quadrangle $[u_1, u_2] \times [v_1, v_2] \subset G$ the four corners $x(u_i, v_j)$ are coplanar, and 
\item[(B)]
  for every quadrangle $[u_1, u_2] \times [v_1, v_2] \subset G$ the tangent planes at the four corners $x(u_i, v_j)$ are concurrent.
\end{itemize}
were classified in the following theorem.

\begin{theorem}\cite[Thm.~4]{Degen:1994:PlaneSilhouettes1}
  \label{thm:Degen}
  Nets with both properties (A) and (B) are conjugate and have plane silhouette lines. Furthermore, for both families of net curves, their planes belong to a pencil. The axes of these pencils are the loci of the vertices of the enveloping cones.
\end{theorem}

%%% Local Variables:
%%% mode: latex
%%% TeX-master: "nets"
%%% End:

%% file: smoothextensions.tex
\section{Piecewise smooth extension by projective translation surfaces}
\label{sec:piecewise_smooth_extension}

Projective translation surface patches are the basis of our piecewise smooth extension.
These surfaces include but are not limited to Dupin cyclides and supercyclides.
Previously, we have used these two families of surfaces to construct piecewise smooth extensions of discrete nets (see~\cite{BobenkoHuhnenVenedey:Cyclides, BobHVRoer:Supercyclides}).

\begin{definition}
  A parametrized surface patch~$f:[0,1]^2 \to \RP^3$ is a \emph{projective translation surface} if there exist two curves $\hc p, \hc q: [0,1] \to \R^4$ such that $f(u,v) = [\hc p(u) + \hc q(v)]$.
\end{definition}

In other words, there exists a choice of homogeneous coordinates for a projective translation surface such that it is a euclidean translation surface in~$\R^4$.
The following theorem provides two equivalent characterizations of projective translation surfaces.

\begin{theorem}[see~\mbox{\cite[Sect.~2]{Degen:1994:PlaneSilhouettes1}}]
  Let $f: [0,1]^2 \to \RP^3$ be a parametrized surface patch which is not contained in a plane.
  Then the following are equivalent:
  \begin{enumerate}
  \item Every parameter rectangle of~$f$ is planar.
  \item The parameter lines of~$f$ are conjugate and the Laplace transforms degenerate to curves.
  \item The surface patch is part of a projective translation surface.
  \end{enumerate}
\end{theorem}

\subsection{Adapted projective translation patch}
\label{sec:apat-proj-transl}

We define a projective translation surface patch adapted to a planar quadrilateral and derive the necessary initial data to determine it uniquely.

\begin{definition}
  \label{defn:adapted-patch}
  Let $x\ind{00}, x\ind{10}, x\ind{11}, x\ind{01} \in \RP^3$ be a planar quadrilateral in~$\RP^3$ and $f: [0,1]^2 \to \RP^3$ be a projective translation surface patch. Then the patch is \emph{adapted to the quadrilateral} if the four corners of the patch coincide with the corners of the quadrilateral, i.e.\
  \[
    f(i,j) = x\ind{ij} \text{, for $i,j \in \{0,1\}$.}
  \]
\end{definition}

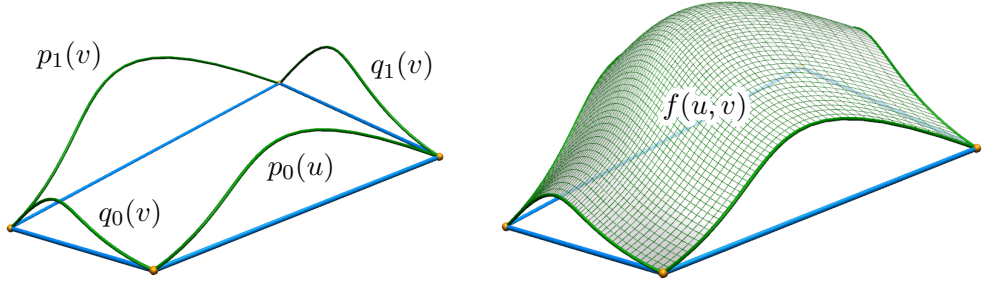
\begin{figure}[t]
  \begin{overpic}[width=.45\linewidth]{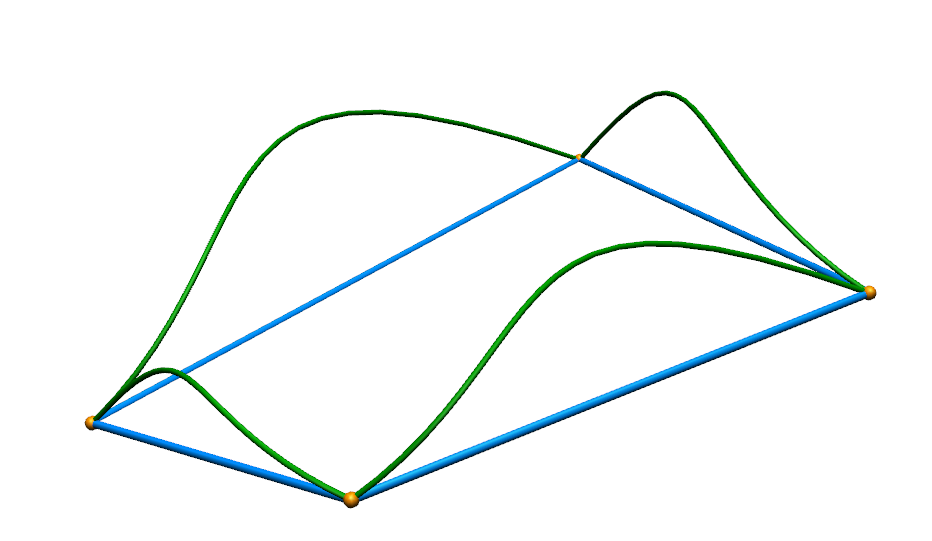}
    \put(27,14){$q_0(v)$}
    \put(80,43){$q_1(v)$}
    \put(60,23){$p_0(u)$}
    \put(15,45){$p_1(v)$}
  \end{overpic}
  \begin{overpic}[width=.45\linewidth]{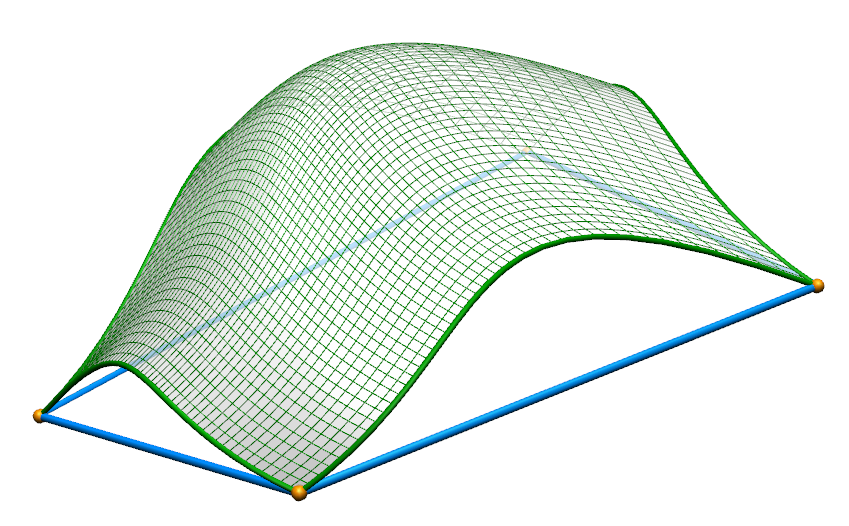}
    \put(35,35){\contour{white}{$f(u,v)$}}
  \end{overpic}
  \caption{Construction of an adapted projective translation surface patch based on two pairs of curves in perspective with respect to the Laplace points (left) and the resulting patch (right).}
  \label{fig:adapted_pt_patch}
\end{figure}

The following Lemma characterizes the necessary and sufficient data that uniquely determines an adapted projective translation surface patch.

\begin{lemma}
  \label{lem:unique-multi-Q-extension-4curves}
  Let $Q \subset \RP^3$ be a planar quadrilateral with Laplace points~$y^1$ and~$y^2$.
  Further let $p_0, p_1, q_0, q_1: [0,1] \to \RP^3$ be four curves with $p_j(i) = q_i(j) = x\ind{i,j}$ for $i,j\in\{0,1\}$ such that
  $p_0$ and $p_1$ ($q_0$ and $q_1$) are in perspective with respect to $y^2$ (resp.\ $y^1$).
  Then there exists a unique adapted projective translation surface patch $f: [0,1]^2 \to \RP^3$ such that
  \begin{itemize}
  \item $f(u,j) = p_j(u)$ for all $u \in [0,1]$ and $j \in \{0,1\}$,
  \item $f(i,v) = q_i(v)$ for all $v \in [0,1]$ and $i \in \{0,1\}$.
  \end{itemize}
\end{lemma}
\begin{proof}
  The two pairs of opposite curves fix a unique homogeneous coordinate for the entire projective translation surface in the following way:
  Since the points of opposite curves $p_0$ and $p_1$, resp. $q_0$ and $q_1$, are in perspective with respect to $y^2$ resp.\ $y^1$, we may choose coordinates in $\R^4$ such that:
  \begin{align*}
    \hc y^1
    & = \hc q_1(v) - \hc q_0(v)
      = \hc x\ind{11} - \hc x\ind{01}
      = \hc x\ind{10} - \hc x\ind{00}
      \text{, for all $v \in [0,1]$, and}\\
    \hc y^2
    & = \hc p_1(u) - \hc p_0(u)
      = \hc x\ind{11} - \hc x\ind{10}
      = \hc x\ind{01} - \hc x\ind{00}
      \text{, for all $u \in [0,1]$.}
  \end{align*}
  The unique projective translation surface patch is then given by
  \[f(u,v)=[\hc p_0(u) + \hc q_0(v) - \hc x\ind{00}]\,.\]
\end{proof}

\subsection{Continuous join of neighboring patches}
\label{sec:cont-join-neighb}

We will discuss the smoothness of two projective translation surface patches joined along a common parameter curve.
In our previous investigations of smooth extensions of discrete circular, asymptotic, and conjugate nets we considered tangent plane continuity.
In case of projective translation surface patches, the tangent planes are spanned by the surface point $f(u,v)$ and the two Laplace points $\laplace{u}f(u,v)$ and $\laplace{v}f(u,v)$.
The Laplace transforms of projective translation surfaces are curves, i.e., $\partial_v\laplace{u}f = \laplace{u}f$ and $\partial_u\laplace{v}f = \laplace{v}f$.
Furthermore, for a projective translation surface patch $f :[0,1] \to \RP^3$ with homogeneous coordinates $f(u,v) = [\hc p(u) + \hc q(v)]$ the Laplace transforms are given by
$\laplace{u}f(u,v) = [\partial_u \hc p(u)]$ and
$\laplace{v}f(u,v) = [\partial_v \hc q(v)]$.

\begin{proposition}
  \label{prop:smooth-join-two-patches}
  Let $f^{00}, f^{10} : [0,1]^2 \to \RP^3$ be two projective translation surface patches with a common coordinate curve $f^{00}(1,v) = f^{10}(0,v)$.
  Then the two surface patches have the same tangent planes along the common curve
  if and only if
  the two pairs of boundary curves $f^{00}(u,0)$ and $f^{10}(u,0)$ resp.\ $f^{00}(u,1)$ and $f^{10}(u,1)$ have a common tangent at $f^{00}(1,0) = f^{10}(0,0)$ resp.\ $f^{00}(1,1) = f^{10}(0,1)$.
\end{proposition}

\begin{proof}
  As noted above, the tangent planes are the planes that contain the surface point and the two Laplace points.
  Since the Laplace transforms degenerate to curves for projective translation surfaces, all $u$-tangents (tangents in $u$-direction) along a $v$-curve intersect in a single Laplace point.
  This Laplace point is already determined by two of the $u$-tangents.
  Thus the Laplace point $\laplace{u}f^{00}(1,v) = \laplace{u}f^{01}(0,v)$ in $u$-direction of the two patches along the common $v$-curve coincide, since they are the intersection of the common $u$-tangents at $f^{00}(1,0) = f^{10}(0,0)$ and $f^{00}(1,1) = f^{10}(0,1)$.

  Since the $v$-tangents along the common $v$-curve also coincide we obtain that the tangent planes spanned by these tangents and the unique Laplace point~$\laplace{u}f^{00}(1,v) = \laplace{u}f^{01}(0,v)$ are equal for the two patches along their common curve.
\end{proof}

\subsection{Smooth extensions with planar parameter curves}
\label{sec:planar-parameter-curves}

We want to show that the extension of Q-nets with supercyclidic patches is a special case of the proposed extension with projective translation surface patches.
Hence we restrict to the family of surfaces with planar parameter curves.
We observe, that the Laplace transforms of projective translation surfaces with planar parameter curves not only degenerate to curves, but are contained in a line.
Furthermore, the supporting planes of the parameter lines of each family pass through the two lines containing the Laplace transforms.
Hence these surfaces are D-surfaces.
Moreover, if the Laplace transforms are parametrized by a quadratic rational function, then the resulting surface is a supercyclide.
In the discrete setting these are discrete $(Q+Q^*)$-nets characterized by~\Cref{thm:qqstarnets}. 
For these surfaces it is possible to formulate a Cauchy type problem for the piecewise smooth extension in the following way.

\paragraph{Cauchy problem for planar parameter curves}
Consider a Q-net $x: \Z^2 \to \RP^3$ with the following additional data:
\begin{enumerate}
\item
  A set of planes $P_x: \Z \to \{ \text{planes in $\RP^3$} \}$ with $(x\ind{0,j}, x\ind{1,j}) \subset P_x(j)$ for all $j \in \Z$.
\item
  A set of planes $P_y: \Z \to \{ \text{planes in $\RP^3$} \}$ with $(x\ind{i,0}, x\ind{i,1}) \subset P_y(i)$ for all $i\in\Z$.
\item
  A continuous differentiable curve $c_x: \R \to \RP^3$ such that $c_x(i) = x\ind{i,0}$ and $c_x([i,i+1])$ is planar. 
\item
  A continuous differentiable curve $c_y: \R \to \RP^3$ such that $c_y(j) = x\ind{0,j}$ and $c_y([j,j+1])$ is planar.
\end{enumerate}
\emph{Claim:} There exists a unique piecewise smooth extension with projective translation surface patches with planar parameter lines such that the surface contains the curves~$c_x$ and $c_y$ and the parameter lines at the edges lie in the prescribed planes~$P_x$ and~$P_y$.

\begin{proof}
  We will first show that the curves and the planes define two line congruences that will be the tangents to the boundary curves of the projective translational surface patches.
  Consider the quadrilateral based at $x\ind{00}$ with the initial data as shown in \Cref{fig:planar-curves-Cauchy-data}.
  Then we project the boundary curves $c_x$ and $c_y$ from the respective Laplace points onto the planes $P_x(1)$ and $P_y(1)$ to obtain the opposite boundary curves.
  Then by \Cref{lem:unique-multi-Q-extension-4curves} we obtain a unique adapted projective translational surface path bounded by these curves.
  The tangents of the projected curves at the vertex~$x\ind{11}$ define the planes at the next edges $x\ind{11}x\ind{21}$ and $x\ind{11}x\ind{12}$, respectively.
  Hence we can continue to project the curves $c_x$ and $c_y$ onto the newly defined planes.
  By construction, the tangents of consecutive pairs of boundary curves coincide along the coordinate polylines.
  Hence by \Cref{prop:smooth-join-two-patches} the tangent planes of neighboring patches coincide along common boundary curves.
\end{proof}

\begin{figure}[t]
  \centering
  \def\tikzscale{.4}
  \input{tikz/planar-curve-Cauchy.tikz}
  \caption{Cauchy data for the construction of a piecewise smooth extension of a Q-net by projective translation surface patches with planar parameter lines.}
  \label{fig:planar-curves-Cauchy-data}
\end{figure}
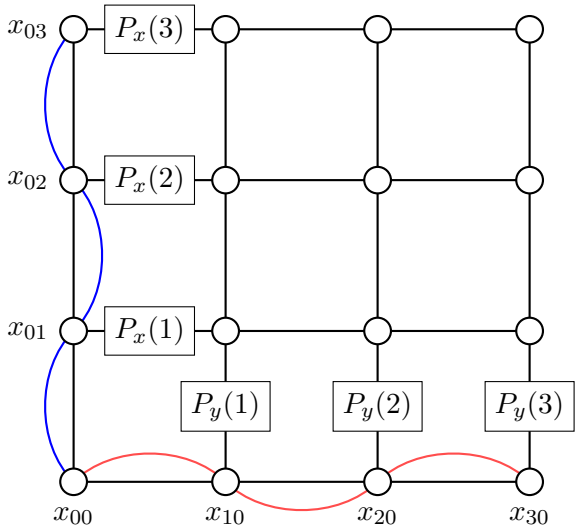

In previous work we constructed supercyclide patches adapted to Q-nets (see~\cite{BobHVRoer:Supercyclides}).
The data prescribed in case of supercyclides is equivalent to the prescription of the four curves of \Cref{lem:unique-multi-Q-extension-4curves}.

\begin{corollary}
  \label{cor:supercyclides}
  Let $Q \subset \RP^3$ be a planar quadrilateral with Laplace points~$y^1$ and~$y^2$.
  Further let $p_0, p_1, q_0, q_1: [0,1] \to \RP^3$ be four conic arcs with $p_j(i) = q_i(j) = x\ind{i,j}$ for $i,j\in\{0,1\}$ such that
  $p_0$ and $p_1$ ($q_0$ and $q_1$) are in perspective with respect to the Laplace points $y^2$ (resp.\ $y^1$).
  Then there exists a unique adapted supercyclide patch $f: [0,1]^2 \to \RP^3$
  \begin{itemize}
  \item $f(u,j) = p_j(u)$ for all $u \in [0,1]$ and $j \in \{0,1\}$,
  \item $f(i,v) = q_i(v)$ for all $v \in [0,1]$ and $i \in \{0,1\}$.
  \end{itemize}
\end{corollary}

\begin{proof}
  For the given initial data there exists a unique projective translation surface patch.
  
  For the supercyclide patch we follow the \emph{construction of adapted SC-patches} in~\cite[Sect.~4]{BobHVRoer:Supercyclides}:
  The pairs of opposite conic arcs define the two line congruences.
  The two tangents at the vertex~$x\ind{11}$ are the unique lines that are defined by the tangents to the curves at the vertices of the coordinate axes.
  Then the two conic arcs on the coordinate axes define a unique supercyclide patch.
  In this patch, the opposite conic sections are in perspective with respect to the Laplace points.

  The supercyclide patch is a projective translation surface patch and by uniqueness of \Cref{lem:unique-multi-Q-extension-4curves} the projective translation surface patch and the supercyclide patch coincide.
\end{proof}

The global construction of supercyclidic nets from two line congruences and two prescribed piecewise conic curves is a special case of the construction presented here.

%%% Local Variables:
%%% mode: latex
%%% TeX-master: "nets"
%%% End:

%% file: tikz/planar-curve-Cauchy.tikz
\begin{tikzpicture}

\def\m{3}
\def\n{3}
\foreach \x in {0,1,...,\m} {
	\foreach \y in {0,1,...,\n} {
		\coordinate (x\x\y) at (2*\x,2*\y);
  }
}

\foreach \x in {0,1,...,\m} {
	\draw[myline] (x\x0)--(x\x\m);
	\draw[myline] (x0\x)--(x\n\x);
}

\foreach \x in {1,...,\m} {
	\node[shape=rectangle, fill=white, draw] (px\x) at ($(x\x0)!.5!(x\x1)$) {$P_y(\x)$};
}

\foreach \x in {1,...,\n} {
	\node[shape=rectangle, fill=white, draw] (py\x) at ($(x0\x)!.5!(x1\x)$) {$P_x(\x)$};
}

\draw[myline, blue] (x00) .. controls +(-.5,.5) and +(-.5,-.5) .. (x01);
\draw[myline, blue] (x01) .. controls +(.5,.5) and +(.5,-.5) .. (x02);
\draw[myline, blue] (x02) .. controls +(-.5,.5) and +(-.5,-.5) .. (x03);
%\draw[myline, blue] (x03) .. controls +(.5,.5) and +(.5,-.5) .. (x04);
%\draw[myline, blue] (x04) .. controls +(-.5,.5) and +(-.5,-.5) .. (x05);

\draw[myline, red!70] (x00) .. controls +(.5,.5) and +(-.5,.5) .. (x10);
\draw[myline, red!70] (x10) .. controls +(.5,-.5) and +(-.5,-.5) .. (x20);
\draw[myline, red!70] (x20) .. controls +(.5,.5) and +(-.5,.5) .. (x30);
%\draw[myline, red!70] (x30) .. controls +(.5,-.5) and +(-.5,-.5) .. (x40);
%\draw[myline, red!70] (x40) .. controls +(.5,.5) and +(-.5,.5) .. (x50);

\foreach \x in {0,1,...,\m} {
	\foreach \y in {0,1,...,\n} {
		 
		\path (x\x\y) [mypoint=white];
}}

\foreach \x in {0,1,...,\m} {
		\node[below=.2] at (x\x0) {$x_{\x0}$};
}
\foreach \y in {1,...,\m} {
		\node[left=.2] at (x0\y)  {$x_{0\y}$};
}

\end{tikzpicture}

%% file: qnets-quadrics.tex
\section{Projective translation nets in quadrics}
\label{sec:trans_in_quadric}

Some special families of nets have an interpretation as special nets in quadrics.
Circular nets, for example, may also be considered as Q-nets in the M\"obius quadric in $\PMoebSpace$.
Hence multi-circular nets correspond to multi-Q-nets in the same quadric.
This motivates the investigation of multi-Q-nets, that is, projective translation nets in quadrics in general. 

Let $\quadric \subset \RP^n$ be a quadric given by a non-degenerate symmetric bilinear form $\bilform{\cdot}{\cdot}$ on~$\R^{n+1}$ and $\hc x\ind{i,j} = \hc{p}\ind{i} + \hc{q}\ind{j} \in \R^{n+1}$, $i,j \in \Z$ be a translation net.
The following proposition shows, that the Laplace transforms of a projective translation net in a quadric have to lie in orthogonal subspaces with respect to the bilinear form.
\begin{proposition}
  \label{prop:projective-translation-in-quadric}
  Let $x = [\hc p\ind{i} + \hc q\ind{j}]$ be a projective translation net in a quadric $\quadric$.
  Then for all $i,j \in \Z$ we have
  \[
  \bilform{\hc p\ind{i+1} - \hc p\ind{i}}{\hc q\ind{j+1} - \hc q\ind{j}} = 0
  \;.
  \]
  In other words, the Laplace transforms
  $\hc y^1\ind{i} = \hc p\ind{i+1} - \hc p\ind{i}$ and
  $\hc y^2\ind{j} = \hc q\ind{j+1} - \hc q\ind{j}$
  lie in polar subspaces.
\end{proposition}
\begin{proof}
  If the net~$x\ind{i,j}$ lies in the quadric $\quadric$ for all $i,j\in\Z$ then
  \begin{align*}
    0 =
    \bilform{\hc x\ind{i,j}}{\hc x\ind{i,j}} 
    = \bilform{\hc p\ind{i}}{\hc p\ind{i}} 
    + 2 \bilform{\hc p\ind{i}}{\hc q\ind{j}} 
    + \bilform{\hc q\ind{j}}{\hc q\ind{j}}
  \end{align*}
  If we consider an elementary quadrilateral with vertices $\hc x\ind{i,j}$, $\hc x\ind{i+1,j}$, $\hc x\ind{i+1,j+1}$, and $\hc x\ind{i,j+1}$ we obtain the following four equations:
  \begin{align*}
    \bilform{\hc p\ind{i}}{\hc p\ind{i}} 
    + 2 \bilform{\hc p\ind{i}}{\hc q\ind{j}} 
    + \bilform{\hc q\ind{j}}{\hc q\ind{j}}
    & = 0 \\
    \bilform{\hc p\ind{i+1}}{\hc p\ind{i+1}} 
    + 2 \bilform{\hc p\ind{i+1}}{\hc q\ind{j}} 
    + \bilform{\hc q\ind{j}}{\hc q\ind{j}}
    & = 0 \\
    \bilform{\hc p\ind{i+1}}{\hc p\ind{i+1}} 
    + 2 \bilform{\hc p\ind{i+1}}{\hc q\ind{j+1}} 
    + \bilform{\hc q\ind{j+1}}{\hc q\ind{j+1}}
    & = 0 \\
    \bilform{\hc p\ind{i}}{\hc p\ind{i}} 
    + 2 \bilform{\hc p\ind{i}}{\hc q\ind{j+1}} 
    + \bilform{\hc q\ind{j+1}}{\hc q\ind{j+1}}
    & = 0 
  \end{align*}
  Taking the appropriate linear combination of the four equations yields
  \begin{align*}
    \bilform{\hc p\ind{i}}{\hc q\ind{j}} 
    - \bilform{\hc p\ind{i+1}}{\hc q\ind{j}} 
    + \bilform{\hc p\ind{i+1}}{\hc q\ind{j+1}} 
    - \bilform{\hc p\ind{i}}{\hc q\ind{j+1}} & = 0 \\ \Leftrightarrow \quad
    \bilform{\hc p\ind{i}}{\hc q\ind{j} - \hc q\ind{j+1}}
    + \bilform{\hc p\ind{i+1}}{\hc q\ind{j+1} - \hc q\ind{j}} & = 0 \\ \Leftrightarrow \quad
    \bilform{\hc p\ind{i+1} - \hc p\ind{i}}{\hc q\ind{j+1} - \hc q\ind{j}} & = 0 \;.
  \end{align*}
\end{proof}

With the above proposition we can generate a projective translation net in a quadric from two polygons $\hc y^1\ind{i}$ and $\hc y^2\ind{j}$ in polar subspaces.
In the next subsection we describe translations and polar reflections with respect to a symmetric bilinear form.
This provides a way to generate multi-Q-nets by orthogonal transformations with respect to a given quadric.

\subsection{Translations and projective reflections}
\label{ssec:translations-and-reflections}

Another way to generate projective translation surfaces in quadrics is by particular sets of reflections instead of translations. This will help us understand the geometry of projective translation surfaces in quadrics.

Let $\quadric \subset \RP^n$ be a quadric given by a non-degenerate symmetric bilinear form $\bilform{\cdot}{\cdot}$ on~$\R^{n+1}$. Then for every point $n \not \in \quadric$ we define the reflection with respect to the point $n$ and the polar hyperplane of $n$ with respect to the quadric as follows:
\begin{align*}
\sigma_n : \R^{n+1} &\longto \R^{n+1} \\
  \hc x  &\longmapsto 
  \hc x - 2 \frac{\bilform{\hc x}{\hc n}}{\bilform{\hc n}{\hc n}} \hc n \,.
\end{align*}
This reflection acts in $\R^{n+1}$ as a translation on points with $\bilform{\hc x}{\hc n} = \mathtt{const}$. This allows us to generate a projective translation net in a quadric using reflections through the following Cauchy data.
\begin{theorem}
\label{thm:reflections-translation-net}
Let $n^1_i$, $i\in \Z$ and $n^2_j$, $j\in\Z$ be two polygons with $\bilform{\hc n^1_i}{\hc n^2_j} = 0$ for all $i,j \in \Z$ and $x\ind{00} \in \quadric$ an initial point. 
Then the two families of reflections $\sigma^1_i = \sigma_{n^1_i}$ and $\sigma^2_j = \sigma_{n^2_j}$ generate a projective translation net $x\ind{ij} \in \quadric$ with homogeneous coordinates 
\[
\hc x\ind{ij} = 
\left(
\prod_{k=0}^{i-1}{\sigma^1_k}\circ
\prod_{l=0}^{j-1}{\sigma^2_l}\right)
(\hc x\ind{00})\,.
\]
Furthermore, the Laplace transforms of a net defined by reflections are $y^1_i = [\hc n^1_i]$ and $y^2_j = [\hc n^2_j]$.
\end{theorem}
\begin{proof}
  First note, that all the points of the net lie on the quadric, because all the reflections map the quadric onto intself.
  Furthermore, the reflections of the two families $\sigma^1_i$ and $\sigma^2_j$ commute as the respective points are in orthogonal subspaces and satisfy $\bilform{\hc n^1_i}{\hc n^2_j} = 0$. 
  For the initial quadrilateral we obtain
  \begin{align*}
    \hc x\ind{10} 
    & = \sigma^1_0 (\hc x\ind{00}) = 
      \hc x\ind{00} 
      - 2 
      \frac{\bilform{\hc x\ind{00}}{\hc n^1_0}}{\bilform{\hc n^1_0}{\hc n^1_0}} \hc n^1_0, \qquad
    \hc x\ind{01}
      = \sigma^2_0 (\hc x\ind{00}) = 
      \hc x\ind{00} 
      - 2 
      \frac{\bilform{\hc x\ind{00}}{\hc n^2_0}}
      {\bilform{\hc n^2_0}{\hc n^2_0}} \hc n^2_0
    \\
    \hc x\ind{11} 
    & = \sigma^1_i \circ \sigma^2_1(\hc x\ind{00})
      = \sigma^1_i (\hc x\ind{01})
      = 
      \hc x\ind{00} 
      - 2 
      \frac{\bilform{\hc x\ind{00}}{\hc n^2_0}}
      {\bilform{\hc n^2_0}{\hc n^2_0}} \hc n^2_0
      - 2
      \frac{\bilform{\hc x\ind{00}}{\hc n^1_0}}
      {\bilform{\hc n^1_0}{\hc n^1_0}} \hc n^1_0
      \,.
  \end{align*}
  We observe, that the translations from $x\ind{00} \to x\ind{10}$ and from $x\ind{01} \to x\ind{11}$ are the same, since $\bilform{\hc n^1_0}{\hc n^2_0} = 0$. But this is true for the entire net as $\bilform{\hc n^1_i}{\hc n^2_j} = 0$ for all $i,j \in \Z$. 

  The Laplace transforms are the polygons given by
  \begin{align*}
    \hc y^1_i = - 2
    \frac{\bilform{\hc x\ind{i0}}{\hc n^1_0}}
    {\bilform{\hc n^1_0}{\hc n^1_0}} \hc n^1_0 \qquad \text{and} \qquad
    \hc y^2_j = - 2 
    \frac{\bilform{\hc x\ind{0j}}{\hc n^2_0}}
    {\bilform{\hc n^2_0}{\hc n^2_0}} \hc n^2_0
    \, .
  \end{align*}
  and hence by~\Cref{thm:multiqnets} the net is a projective translation net. 
\end{proof}

%%% Local Variables:
%%% mode: latex
%%% TeX-master: "nets"
%%% End:

%% file: circularnets.tex
\section{Circular nets}
\label{sec:circularnets}

Circular nets are another important family of nets.
They are discretizations of surfaces parametrized along principal curvature directions.

In the following we will describe nets in~$\Rinfty$. Generically, the nets will not have vertices at infinity.
But for some proofs it is very convenient to map a vertex to infinity and the circles passing through this vertex to lines.
A more homogeneous description can be achieved by using~$\Sphere^3$ instead.

\begin{definition}
  \label{defn:circular-net}
  A 2-dimensional \emph{circular net} is a map $x: \Z^2 \to \Rinfty$ such that the vertices $x\ind{i,j}, x\ind{i+1,j}, x\ind{i+1,j+1}, x\ind{i,j+1}$ lie on a circle.
\end{definition}

\begin{remark*}
  As our investigations are motivated by differential geometry we will only consider \emph{embedded} circular quadrilaterals.
  In the non-embedded case, there does not exist a smooth limit, which is of central importance in discrete differential geometry.
  As shown in~\cite{BobenkoMatthesSuris:approximation}, the smooth limits of embedded circular nets are surfaces parametrized along principal curvature lines.
  This analogy is apparent comparing Theorems~\ref{thm:circular} and~\ref{thm:multi-circular-smooth}.
\end{remark*}

As in the case of $Q$- and $Q^*$-nets we define multi-circular nets in the following way.

\begin{definition}
  \label{defn:multi-circular-net}
  A 2-dimensional \emph{multi-circular net} is a map $x: \Z^2 \to \Rinfty$ such that for every $i_0 \neq i_1$ and $j_0 \neq j_1$ the coordinate quadrilaterals $x\ind{i_0, j_0}, x\ind{i_0, j_1}, x\ind{i_1, j_1}, x\ind{i_1, j_0}$ are circular.
\end{definition}

An example of a multi-circular net is a discrete rotational surface shown in \Cref{fig:multi-circular-net}.
The circularity of quadrilaterals is preserved by M\"obius transformations.
Hence inverting a rotational surface in a sphere yields another multi-circular net.
As we will show in \Cref{thm:circular}, this is already the most general kind of multi-circular net.
In particular, Dupin cyclides sampled along the principal curvature lines yield multi-circular nets.

\begin{figure}[tb]
  \centering
  \includegraphics[width=.5\linewidth]{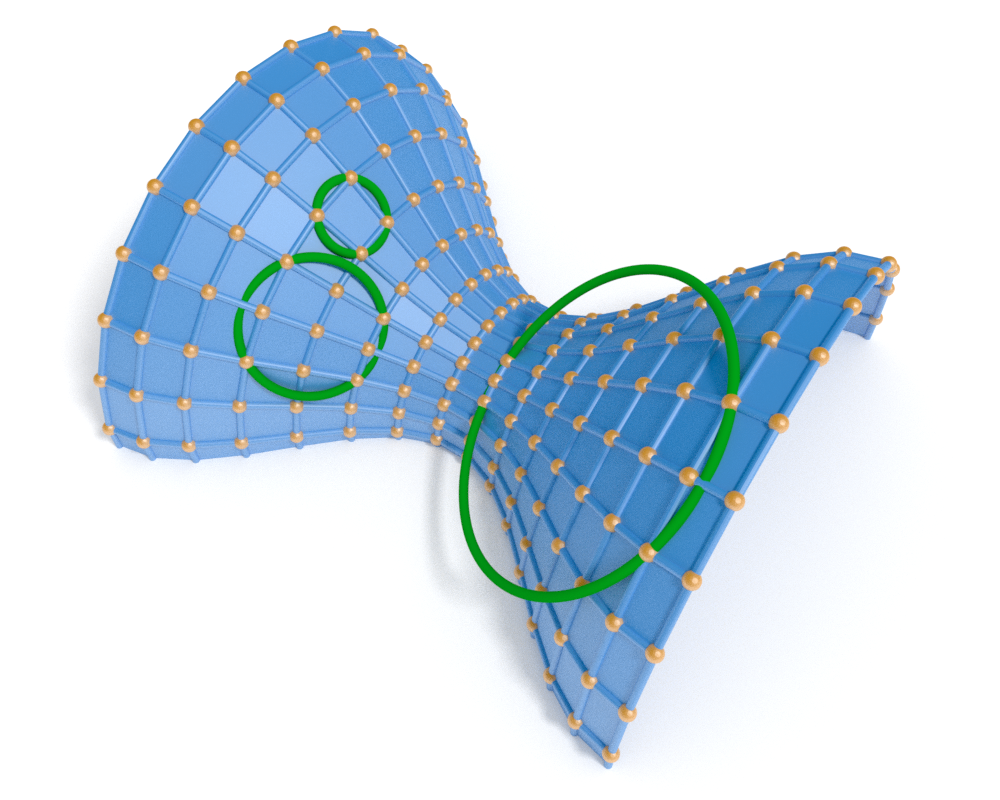}
  \input{tikz/moebspace.tikz}
  \caption{
    Left: A rotational net is a multi-circular net.
    Right: The space~$\PMoebSpace$ is used to describe the projective model of M\"obius geometry. It naturally contains the elementary~$\Rinfty$ (paraboloid) and spherical~$\Sphere^3$ model. The figure shows the respective models in~$\R^{2,1}$.
  }
  \label{fig:multi-circular-net}
\end{figure}

In the following we present two perspectives on circular nets: 
In Section~\ref{ssec:circular_elementary} we describe some elementary properties of multi-circular nets in Euclidean 3-space.
In the special case of the M\"obius quadric in~$\PMoebSpace$, circular and multi-circular nets are Q- resp.\ multi-Q-nets in a quadric in the sense of the previous Section~\ref{sec:trans_in_quadric}.
This yields a projective point of view on multi-circular nets presented in Section~\ref{ssec:circular-nets-as-q-nets}.

\subsection{Elementary properties}
\label{ssec:circular_elementary}

As stated in the Remark at the beginning of \Cref{sec:circularnets} we are interested in embedded circular nets only.
Whether or not a circular quadrilateral is embedded is invariant with respect to M\"obius transformations, in particular, sphere inversions. 
The following Lemma assures that the ``large'' circles of a multi-circular net are embedded as well if no two adjacent circular quadrilaterals have the same circumcircle, which we always assume without explicitly stating it.

\begin{lemma}
  \label{lem:embedded-multi-circular}
  Let $x: \Z^2 \to \Rinfty$ be a multi-circular net. Then all parameter quadrilaterals are embedded.
\end{lemma}
\begin{proof}
  Consider two neighboring circular quadrilaterals of a multi-circular net as shown in \Cref{fig:orthogonal-spheres} (left).
  An inversion in a sphere centered at $x\ind{i+1,j+1}$ maps the two circumcircles onto two straight lines.
  From the ordering of points we can deduce that the circular quadrilateral $x\ind{i,j}, x\ind{i+1,j}, x\ind{i+1,j+2}, x\ind{i,j+2}$ is embedded.

  Now it follows by induction that all parameter quadrilaterals are embedded.
\end{proof}

Multi-circular nets are special multi-$Q$-nets, so all edges across a quadrilateral strip intersect in the unique Laplace point.
In the circular case, there exists an additional role of the Laplace points.
Consider spheres centered at the Laplace points and orthogonal to the circumcircle of a quadrilateral.
The inversions in these spheres map the respective opposite edges of the circular quadrilaterals onto each other (and the circumcircle onto itself), see \Cref{fig:orthogonal-spheres}. 
Since the Laplace points coincide along coordinate strips we obtain the following Lemma, see \Cref{fig:orthogonal-spheres} (left).

\begin{lemma}
  \label{lem:strip-sphere}
  For a generic strip of a multi-circular net, there exists a sphere orthogonal to all circles of the strip.
  The inversion in this sphere maps the two bounding parameter polylines onto each other.
\end{lemma}

Looking at a single quadrilateral we further observe that the spheres of distinct coordinate directions are orthogonal, see \Cref{fig:orthogonal-spheres} (right).

\begin{lemma}
  \label{lem:ortho-spheres}
  The orthogonal spheres of distinct coordinate directions are orthogonal to each other.
\end{lemma}
\begin{proof}
  Since the centers of the spheres lie in the plane of the quadrilateral, the spheres are orthogonal to the plane.
  Hence, we only need to consider the configuration of circles in the plane.
  If we map one of circles to a straight line by an inversion, the quadrilateral becomes a trapezoid and one of the orthogonal circles becomes a line.
  This line is orthogonal to the other symmetry circle of the trapezoid.
\end{proof}

\begin{figure}[tb]
  \centering
  \def\tikzscale{.6} \input{tikz/multi_circular_net.tikz}
  \hspace{1ex}
  \def\tikzscale{.6} \input{tikz/multi_circular_orthogonal_circles.tikz}
  \caption{
    Common orthogonal sphere to neighboring circular quadrilaterals (left). 
    Orthogonality of spheres for different directions (right).
  }
  \label{fig:orthogonal-spheres}
\end{figure}
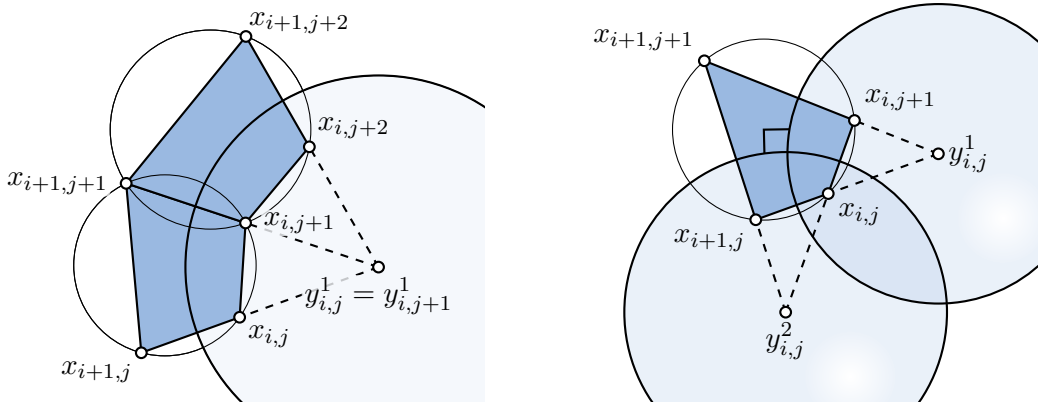

The idea is to use the orthogonal spheres of Lemma~\ref{lem:ortho-spheres} to classify multi-circular nets.
To understand the different families of orthogonal spheres that exist, we move on to a more elaborate model of sphere geometry.

\subsection{Circular nets as Q-nets in quadrics}
\label{ssec:circular-nets-as-q-nets}

We consider the projective model of M\"obius geometry in~$\PMoebSpace$.
A very brief introduction of the necessary projective dictionary of M\"obius geometry is given in~\cite[Sect.~9.3]{BobenkoSuris:DDG}.
A detailed geometric description of M\"obius geometry can be found in Blaschke~\cite[Chp.~2,3,6]{Blaschke:DiffGeoIII}.

The space $\MoebSpace$ is the real vector space~$\R^5$ equipped with the following inner product
\[
  \bilform{\hc x}{\hc y} =
  \hc x_1 \hc y_1 + \hc x_2 \hc y_2 + \hc x_3 \hc y_3
  + \hc x_4 \hc y_4 - \hc x_5 \hc y_5.
\]
The M\"obius quadric in the corresponding projective space~$\PMoebSpace$ is given as usual by
$\quadric = \{ [\hc x] \in \PMoebSpace \,|\, \bilform{\hc x}{\hc x} = 0 \}$ (see \Cref{fig:multi-circular-net}, left).
With respect to the affine coordinates $(\tfrac{\hc x_i}{\hc x_5})_{i=1}^4 \in \R^4$ the quadric is the unit $3$-sphere.
Hence, if we have a (multi-)$Q$-net in the quadric then it corresponds to a (multi-)circular net in the $3$-sphere (see \Cref{fig:multi-circular-net}, right).
The points~$s$ outside the quadric, that is with $\bilform{\hc s}{\hc s} > 0$ represent the 2-spheres in~$\Sphere^3$ obtained from the intersection of~$\Sphere^3$ with the polar plane $\{ [\hc x] \in \quadric\,|\, \bilform{\hc x}{\hc s} = 0 \}$ of~$s$.
The projective reflections discussed in Section~\ref{ssec:translations-and-reflections} become sphere inversions in the respective $2$-spheres.
If we have two spheres represented by points~$s_1, s_2 \in \PMoebSpace$ then the corresponding spheres intersect orthogonally if $\bilform{\hc s_1}{\hc s_2} = 0$.

In Section~\ref{sec:trans_in_quadric} we have seen that the Laplace transforms of projective translation nets in quadrics lie in orthogonal subspaces.
Furthermore, these Laplace transforms can be used to generate the nets by reflections from an initial point on the quadric as stated in \Cref{thm:reflections-translation-net}.
So we will classify multi-circular nets by the families of orthogonal spheres and the corresponding orthogonal subspaces in the projective model $\PMoebSpace$ of M\"obius geometry.

\begin{theorem}
  \label{thm:circular}
  Let $x: \Z^2 \to \Sphere^3$ be a multi-circular net. Then it is obtained by a M\"obius transformation of one of the following:
  \begin{enumerate}
  % \item a single strip,
  \item a surface of revolution,
  \item a cone, or
  \item a cylinder.
  \end{enumerate}
\end{theorem}
\begin{proof}
  If~$x$ is a multi-circular net on~$\Sphere^3$ it is lifted to a multi-Q-net in the M\"obius quadric in~$\PMoebSpace$ (see~\cite[Sect.~9]{BobenkoSuris:DDG}).
  By \Cref{prop:projective-translation-in-quadric} the Laplace transforms~$\hc y^1$ and $\hc y^2$ of the net lie in polar subspaces with respect to the M\"obius inner product.
  We denote the subspaces spanned by the Laplace transforms by
  $U^1 = \lspan{\hc y^1_i| i \in \Z}$ and $U^2 = \lspan{\hc y^2_j\,|\, j \in \Z}$.
  Additionally, by \Cref{thm:reflections-translation-net}, the whole net is generated by the reflections induced by the points of the Laplace transforms.
  As we are only considering embedded circular quadrilaterals, all Laplace points lie outside the quadric, i.e., the signature is~$({+})$.
  Hence we can classify the nets depending on the signatures of the subspaces $(U^1, U^2)$ of~$\MoebSpace$.
  As the dimension of~$\MoebSpace$ is five, at least one of the two subspaces has dimension at most two.
  So we consider the signatures of subspaces of dimensions as most two and their complements:

  If one of the spaces is 1-dimensional, then one family of parameter lines is generated by a single reflection.
  Then the two parameter lines are symmetric with respect to a single sphere.
  Hence the entire net only consists of a single strip which is contained in a cone.

  If one of the spaces is 2-dimensional, then we have to distinguish the following cases:
  \begin{itemize}
  \item $({+}{+})$: %
    The subspace of signature $({+}{+})$ is a pencil of spheres passing through a common circle. Looking at a representation in Euclidean space, this circle can be transformed into a line and the spheres passing through the circle become planes passing through this line. Hence the surfaces generated by the reflections in these planes are M\"obius equivalent to \emph{surfaces of revolution}.
  \item $({+}{-})$: %
    The subspace of signature $({+}{-})$ is a hyperbolic pencil of spheres
    (orthogonal to the set of spheres passing through two points). By
    a M\"obius transformation these spheres can be normalized to
    concentric spheres. Hence the surfaces are M\"obius equivalent to \emph{cones}.
  \item $({+}0)$: %
    The subspace of signature $({+}0)$ corresponds to a pencil of spheres
    tangent to a plane at a common point. By a M\"obius transformation
    this pencil can be transformed to a family of parallel
    planes. Hence the surfaces are M\"obius equivalent to \emph{cylinders}.
  \end{itemize}
\end{proof}

The exact same family of smooth surfaces is characterized in the Theorem of Vessiot~\cite[p.~132]{HertrichJeromin:MoebiusDiffGeo} stating the isothermic channel surfaces are M\"obius transforms of cones, cylinders, or surfaces of revolution.
\emph{Discrete isothermic nets} are circular nets with the additional property, that every vertex and its four diagonal neighbors lie on a common sphere, introduced in~\cite{Bobenko:isothermic}. 
This is obviously satisfied by multi-circular nets, since the four diagonal neighbors lie on a common circle.
So we have the following Proposition.
\begin{proposition}
  Multi-circular nets are discrete isothermic.
\end{proposition}

\paragraph{Cauchy data for multi-circular nets}
Multi-circular nets can be constructed using \Cref{thm:reflections-translation-net} and the following initial data:
\begin{itemize}
\item
  Let~$x_{00} \in \Rinfty$ and $s^1_i$ and $s^2_j$ ($i,j \in \Z$) be two families of spheres, such that $s^1_i$ and $s^2_j$ intersect orthogonally for all $i,j \in \Z$.
\end{itemize}

Then there exists a unique multi-circular net such that the circles through the vertices $(x\ind{i,j}, x\ind{i+1,j}, x\ind{i+1,j+1}, x\ind{i,j+1})$ of the quadrilaterals are orthogonal to the spheres~$s^1_i$ and~$s^2_j$.
This net is generated by reflecting the initial point~$x_{00}$ in the respective spheres.
That is, $x\ind{i+1,j}$ is the image of~$x\ind{i,j}$ of the inversion in~$s^1_i$ and  $x\ind{i,j+1}$ is the image of~$x\ind{i,j}$ of the inversion in~$s^2_j$.
The two inversions commute, since the spheres of distinct directions intersect orthogonally.
%A 2-dimensional example is shown in \Cref{fig:multi-circular-generation}.
% \begin{figure}[bt]
%   \includegraphics[width=.6\textwidth]{blender/moebius-generation.png}
%   \caption{Construction of a planar multi-circular net from an initial point and two families of orthogonal circles.} 
%   \label{fig:multi-circular-generation}
% \end{figure}

\subsection{Piecewise smooth extensions}
\label{sec:piecewise-smooth-ext}

We start this section with a theorem that characterizes smooth parametrized surfaces with circular parameter rectangles. 
These surfaces are a suitable circular analog of projective translation surface patches.
We will use them for the extension of circular nets.
These patches are slightly more general than the Dupin cyclide patches used in~\cite{BobenkoHuhnenVenedey:Cyclides}.

\begin{theorem}
  \label{thm:multi-circular-smooth}
  Let $x: [0,1]^2 \to \R^3$ be a parametrized surface patch such that every parameter quadrilateral is circular.
  Then the patch is parametrized along principal curvature lines and is a M\"obius transform of one of the following:
  \begin{enumerate}
  \item a surface of revolution,
  \item a cone, or
  \item a cylinder.
  \end{enumerate}
  Furthermore, every pair of principal curvature lines of the same direction can be mapped onto each other by sphere inversions.
\end{theorem}
\begin{proof}
  By sampling the given surface~$x$ along the parameter lines, we obtain a multi-circular net.
  By \Cref{thm:circular} the sampled net is M\"obius equivalent to a discrete para\-metrized piece of a surface of revolution, a cone, or a cylinder.
  In particular there exist two families of sphere inversions that map the discrete parameter lines onto each other.
  This property is preserved if we consider parameter curves instead of the sampled polygons.
  Hence for each pair of $u$- resp.\ $v$-lines, there exists a sphere inversion, that maps the pairs of lines onto each other.
  Furthermore, for a quadrangle bounded by parameter lines, the $u$- and the $v$- sphere are orthogonal to each other.
  
  Taking the limit of the quadrangle in one direction shows, that the parameter curves lie on a sphere, that is, are spherical curves.
  Taking the limit in the other direction preserves the orthogonality and shows that the two spherical curves intersect orthogonally at the limit point of the parameter quadrangle.

  Surfaces with planar net quadrangles are projective translation surfaces parametrized along conjugate parameter lines (see~\cite[Sect.~2]{Degen:1994:PlaneSilhouettes1}).
  In our case, the parameter curves intersect orthogonally, and hence the conjugate parametrization is indeed a curvature line parametrization.

  Now we can use the same arguments as in the discrete case to classify the possible surfaces in terms of the families of orthogonal spheres.
  %reference: Blaschke, DiffGeo III, \S 88, 
\end{proof}

We call the surfaces characterized in the above \Cref{thm:multi-circular-smooth} \emph{M\"obius translation surfaces}.
Similar to \Cref{defn:adapted-patch} we define adapted patches for circular nets in M\"obius geometry.

\begin{definition}
  \label{defn:apated-circular-patch}
  Let $x\ind{00}, x\ind{10}, x\ind{11}, x\ind{01} \in \Rinfty$ be a circular quadrilateral and $f: [0,1]^2 \to \Rinfty$ be a M\"obius translation surface patch.
  Then the patch is \emph{adapted to the quadrilateral} if the four corners of the patch coincide with the corners of the quadrilateral, i.e.\
  \[
    f(i,j) = x\ind{ij} \text{, for $i,j \in \{0,1\}$.}
  \]
\end{definition}

We want to construct a M\"obius translation surface patch adapted to a given circular quadrilateral.

% The sphere containing the spherical curve has to intersect the sphere centered at~$y^2$ mapping the opposite edges of 1-direction onto each other. Otherwise the corresponding pencils do not have the correct signatures.

\begin{lemma}
  \label{lem:circular-adapted-patch}
  Let $x:\{0,1\}^2 \to \R^3$ be a circular quadrilateral.
  Further let $p_0:[0,1] \to \R^3$ be a spherical curve on a sphere~$Q_u$ with $p_0(i) = x\ind{i,0}$, $(i \in \{0,1\})$ and
  $q_0: [0,1] \to \R^3$ a circular arc with $q_0(j) = x\ind{0,j}$, $(j \in \{0,1\})$ such that the arc~$q_0$ is orthogonal to~$Q_u$.
  Then there exists a unique M\"obius translation surface patch adapted to the quadrilateral~$x$ with $f(u,0) = p_0(u)$ and $f(0,v) = q_0(v)$.
\end{lemma}
\begin{proof}
  % The orthogonality of the curves and hence of the curve~$\gamma_y$ and the sphere~$S_x$ is necessary, because the parameter lines of M\"obius translation surfaces are curvature lines.
  % The second condition makes sure, that the signatures of the sphere complexes are correct.
  Let $R^2$ be the sphere centered at the Laplace point $y^2$ orthogonal to the circle of the quadrilateral.
  Depending on the intersection of~$Q_u$ and~$R^2$ we can normalize the configuration by an inversion such that~$Q_u$ and~$R^2$ become:
  \begin{enumerate}
  \item
    \label{enum:++}
    two planes through a common line, if $Q_u \cap R^2$ is a circle,
  \item
    \label{enum:+-}
    two concentric spheres, if $Q_u \cap R^2$ is empty, or
  \item
    \label{enum:+0}
    two parallel planes, if $Q_u \cap R^2$ is a point.
  \end{enumerate}
  In case~\eqref{enum:++}, normalization yields two planes $P^u$ and $P^2$ through a common line~$\ell=P^u \cap P^2$.
  (The images of the curve and the circular arc after normalization will be denoted~$p_0$ and $q_0$.)
  Since after the normalization the planes~$P^u$ and~$P^2$ are (still) orthogonal to the circular arc~$q_0$, the axis~$\ell$ passes through the center of the circle defined by~$q_0$.
  By rotating the curve~$p_0$ along~$q_0$ around the axis~$\ell$ we obtain the M\"obius translation surface adapted to the normalized quadrilateral.
  By similar arguments, cases~\eqref{enum:+-} and~\eqref{enum:+0} yield a cone and a cylinder, respectively.
  
  Reversing the normalization yields the desired surface patch.
  
\end{proof}

The above Lemma exhibits an asymmetry between the $u$- and the $v$-direction due to the asymmetry of the M\"obius translational surface patches.
In the following we will focus on special M\"obius translational surface patches given by two circular arcs.

\begin{corollary}
  \label{thm:dupin-cyclide-extension}
  Let $x:\{0,1\}^2 \to \R^3$ be a circular quadrilateral.
  Further let $p_0:[0,1] \to \R^3$ and $q_0: [0,1] \to \R^3$ be circular arcs
  with $p_0(i) = x\ind{i,0}$, $(i \in \{0,1\})$ and
  $q_0(j) = x\ind{0,j}$, $(j \in \{0,1\})$
  that intersect orthogonally at $x\ind{0,0}$.
  Then there exists a unique Dupin cyclide patch adapted to the quadrilateral~$x$.
\end{corollary}
\begin{proof}
  The two circular arcs define two one parameter families (pencils) of spheres containing the arcs.
  In the pencil through $p_0$ (resp. $q_0$) there exists exactly one sphere~$Q_u$ (resp.\ $Q_v$) that is orthogonal to the other circular arc.
  By Lemma~\ref{lem:ortho-spheres} there exist two spheres~$R^1$ and~$R^2$ centered at the Laplace points~$y^1$ and $y^2$ with orthogonal intersection.
  We will show that at least one of the spheres $Q_u$ or $Q_v$ intersects the corresponding Laplace sphere~$R^2$ or $R^1$.
  The intersection will allow us to construct a M\"obius transformation of a rotational surface adapted to the circular quadrilateral, which is the unique Dupin cyclide.
  
  Mapping the point~$x\ind{0,0}$ to infinity by a M\"obius transformation, the spheres~$Q_u$ and $Q_v$ become planes that intersect orthogonally and the circumcircle of the quadrilateral becomes a line intersecting these planes in~$x\ind{1,0}$ and~$x\ind{0,1}$, respectively.
  Additionally, the centers of the spheres~$R^1$ and~$R^2$ become $x\ind{1,0}$ and $x\ind{0,1}$, respectively, and the circles containing~$p_0$ resp.\ $q_0$ become lines through~$x\ind{1,0}$ resp.\ $x\ind{0,1}$.
  In the plane spanned by the line~$p_0$ and the point~$x\ind{0,1}$ we see that one of the Laplace spheres~$R^2$ or~$R^1$ intersects the corresponding plane~$Q_u$ or~$Q_v$.

  Without loss of generality, we assume that~$R^2 \cap Q_u \neq \emptyset$.
  Mapping a point on the intersection to infinity allows us to construct a rotational surface patch adapted to the circular quadrilateral.
  Since the generating meridian is the M\"obius transform of a circular arc the resulting patch of the surface of revolution is a piece of a Dupin cyclide.
  
  This patch is unique by~\Cref{lem:circular-adapted-patch} and is exactly the Dupin cyclide patch used in~\cite{BobenkoHuhnenVenedey:Cyclides}.
\end{proof}

Now consider adapted surface patches of two neighboring quadrilaterals.
To have a continuous join, the two adapted surface patches need to share a common boundary spherical curve.
For each of the patches, there exists a sphere containing the boundary curve, that intersects the resp.\ patch orthogonally.
If the spheres of the two patches coincide, then the two curves even share the same tangent planes.

This allows us to construct smooth extensions of circular nets from Cauchy data.

\begin{figure}[t]
  \includegraphics[width=.4\linewidth]{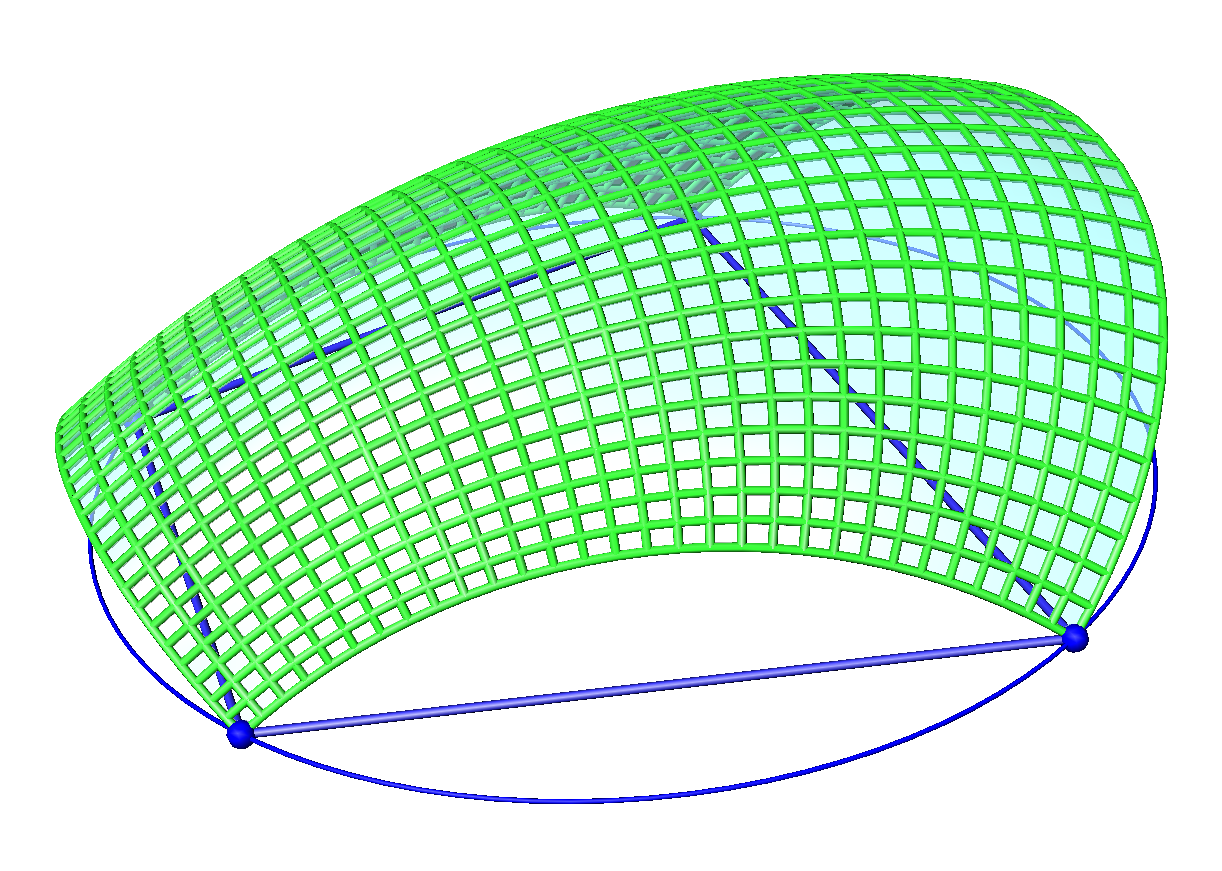}
  \caption{A Dupin cyclide patch adapted to a circular quadrilateral. All the curvature lines are circular arcs. The circular arcs bounding the patch are related by inversions in spheres orthogonal to the circumcircle of the quadrilateral.}
  \label{fig:dupin-cyclide}
\end{figure}

\paragraph{Cauchy data for smooth extension of circular nets}

Given a circular net with $\mathcal{C}^1$-circular arc splines along the coordinate axes that intersect orthogonally at the origin.
Then we can construct a unique piecewise smooth extension by Dupin cyclide patches.

Every circular quadrilateral gives rise to two reflections by~\Cref{lem:ortho-spheres} that may be used to propagate the circular arc splines to all edges of the entire circular net.

\emph{Claim:} The~$\mathcal{C}^1$-condition at the vertices is preserved by the propagation of the circular arc splines:
\begin{proof}
Without loss of generality, consider two circular arcs of the first direction of neighboring quadrilaterals with common tangent at the common vertex.
Then there exists a unique sphere through the two common vertices which is orthogonal to this common tangent.
The reflections in the two Laplace spheres of the second direction are orthogonal to this sphere.
Hence the reflected circular arcs are both orthogonal to this sphere and their tangents pass through its center.
So the tangents have to coincide and the reflected arcs join nicely.
\end{proof}

With all circular arcs associated with the edges of the net we can construct the patches for all the quadrilaterals by~\Cref{thm:dupin-cyclide-extension}.
The intersections of the tangents to opposite circular arcs are the centers of the orthogonal spheres containing the boundary curves.
As remarked above, this implies, that the neighboring patches have the same tangent planes along the common boundary curve.
Hence all the patches join nicely and we obtain a piecewise smooth cyclidic net.

%%% Local Variables:
%%% mode: latex
%%% TeX-master: "nets"
%%% End:

%% file: tikz/moebspace.tikz
\begin{tikzpicture}[scale=\tikzscale]

%Elliptic Paraboloid
\begin{scope}%[shift={(2,-5)}]
  \begin{axis}[hide axis, view = {-10}{10}] 

% the top circle
\addplot3 [data cs=polar, domain=0:360, samples=80, y domain=0:0, thick]
		(x, 2, 2);
\addplot3 [data cs=polar, domain=0:360, samples=80, y domain=0:0, thick, blue]
		(x, 1, 1);
\addplot3 [data cs=polar, domain=0:360, samples=80, y domain=0:0, thick]
(x, 1, -1);
\edef\ang{-17}

\addplot3 [
domain=0:1,
samples=10,
y domain=0:0,
thick]
({-cos(\ang)*(1-x) + 2*x*cos(\ang)},
{-sin(\ang)*(1-x)+2*x*sin(\ang)},
{-1.0*(1-x)+2*x});

\edef\ang{180}

\addplot3 [
domain=0:1,
samples=10,
y domain=0:0,
thick]
({-cos(\ang)*(1-x) + 2*x*cos(\ang)},
{-sin(\ang)*(1-x)+2*x*sin(\ang)},
{-1.0*(1-x)+2*x});

\addplot3 [domain=-{sqrt(3)}:{sqrt(3)}, samples=80, y domain=0:0, thick, blue!50!red]
		({-0.5*(1-x^2)}, x, {0.5*(1+x^2)});

 \end{axis}
\end{scope}

\end{tikzpicture}

%%% Local Variables: 
%%% mode: latex
%%% TeX-master: "../nets"
%%% fill-column: 80
%%% End: 

%% file: tikz/multi_circular_net.tikz
\begin{tikzpicture}[scale=\tikzscale]
\clip (-5,-3) rectangle (7,7);

\def\r{2}

\def\rr{2.2}

\coordinate (C1) at (0,0);
\coordinate (C2) at (1,3);

\draw[name path=circle1] (C1) circle (\r);
\draw[name path=circle2] (C2) circle (\rr);

\coordinate (xij) at ($(C1) + (-35:\r)$);
\coordinate (xi1j) at ($(C1) + (255:\r)$);

\path[name intersections={ of = circle1 and circle2}];
\coordinate (xij1) at (intersection-1);
\coordinate (xi1j1) at (intersection-2);

\coordinate (y1) at (intersection of xij--xi1j and xij1--xi1j1);
\coordinate (y2) at (intersection of xij--xij1 and xi1j--xi1j1);
\coordinate (z) at (intersection of xij--xi1j1 and xi1j--xij1);
\path[name path=line1] (y2)--(z);
\path[name intersections={ of = circle1 and line1}];
\coordinate (p1) at (intersection-1);
\coordinate (p2) at (intersection-2);

\coordinate (xij2) at ($(C2) + (-10:\rr)$);
\coordinate (xi1j2) at ($(C2) + (115:\rr)$);

\path[name path=line2] (y1)--($(y1)!5!(xij2)$);
\path[name intersections={ of = circle2 and line2}];
\coordinate (xij2) at (intersection-2);
\coordinate (xi1j2) at (intersection-1);

\node [draw, fill=lightblue!10] at (y1) [circle through={(p1)}] {};
\draw[myline, fill=lightblue] (xij)--(xi1j)--(xi1j1)--(xij1)--cycle;
\draw[myline, dashed] (y1)--(xij2);
\draw[myline, dashed] (y1)--(xij1);
\draw[myline, dashed] (y1)--(xij);

\draw[myline, fill=lightblue] (xij2)--(xi1j2)--(xi1j1)--(xij1)--cycle;
\node [draw, myline] at (y1) [circle through={(p1)}] {};

\draw(C1) circle (\r);
\draw(C2) circle (\rr);

\node[mylabel, fill=lightblue!10] at (xij) [below right=.1] {$x\ind{i,j}$};
\node[mylabel] at (xi1j) [below left=.1] {$x\ind{i+1,j}$};
\node[mylabel, fill=lightblue!10] at (xij1) [right=.2] {$x\ind{i,j+1}$};
\node[mylabel, fill=lightblue!10] at (xij2) [above right=.1] {$x\ind{i,j+2}$};
\node[mylabel] at (xi1j1) [left=.2] {$x\ind{i+1,j+1}$};
\node[mylabel] at (xi1j2) [above right] {$x\ind{i+1,j+2}$};
\node[mylabel, fill=lightblue!10] at (y1) [below=.1] {$y^1\ind{i,j}=y^1\ind{i,j+1}$};

\foreach \p in {xij, xi1j, xi1j1, xij1, xij2, xi1j2, y1}
 	\path (\p) [mypoint=white];

\end{tikzpicture}

%% file: tikz/multi_circular_orthogonal_circles.tikz
\begin{tikzpicture}[scale=\tikzscale]
%\draw(-4,-6) rectangle (6,4);
\clip(-5,-6) rectangle (6,4);

\def\r{2}
\def\rr{2.5}

\coordinate (C1) at (0,0);
\coordinate (C2) at (1,2.5);

\path[name path=circle1] (C1) circle (\r);
\path[name path=circle2] (C2) circle (\rr);

\coordinate (xij) at ($(C1) + (-45:\r)$);
\coordinate (xi1j) at ($(C1) + (265:\r)$);

\path[name intersections={ of = circle1 and circle2}];
\coordinate (xij1) at (intersection-1);
\coordinate (xi1j1) at (intersection-2);

%compute laplace points and intersection of diagonals
\coordinate (y1) at (intersection of xij--xi1j and xij1--xi1j1);
\coordinate (y2) at (intersection of xij--xij1 and xi1j--xi1j1);
\coordinate (z) at (intersection of xij--xi1j1 and xi1j--xij1);

\path[name path=line1] (y2)--(z);
\path[name intersections={ of = circle1 and line1}];
\coordinate (p1) at (intersection-1);

\path[name path=line2] (y1)--(z);
\path[name intersections={ of = circle1 and line2}];
\coordinate (p2) at (intersection-1);

\veclength{$(p1)-(y1)$}
\edef\radius{\VecLength};
\path[name path = ocircle1, fill=lightblue, opacity=.2] (y1) circle (\radius);

\veclength{$(p2)-(y2)$}
\edef\rradius{\VecLength};
\path[name path = ocircle2, fill=lightblue, opacity=.2] (y2) circle (\rradius);

\shade[inner color=white, outer color=lightblue!20] ($(y1)+(-45:2)$) circle (1.0);
\shade[inner color=white, outer color=lightblue!20] ($(y2)+(-45:2)$) circle (1.0);

\draw[myline, dashed] (y1)--(xij1);
\draw[myline, dashed] (y1)--(xij);

\draw[myline, dashed] (y2)--(xi1j);
\draw[myline, dashed] (y2)--(xij);

\draw[myline, fill=lightblue] (xij)--(xi1j)--(xi1j1)--(xij1)--cycle;

\path[name intersections={ of = ocircle1 and ocircle2, by={i1,i2}}];

\begin{scope}
\def\anglesize{.5cm}
\path[name path=anglesymbol] (i1) circle (\anglesize) {};
\path[name intersections={ of =ocircle1 and anglesymbol, by={as1,as2}}];
\path[name intersections={ of =ocircle2 and anglesymbol, by={as3,as4}}];
\veclength{$(as1)-(as4)$}
\path[name path=angleconstruction] ($(as1)!.5!(as4)$) circle (\VecLength/2);
\path[name path=perpbisector] ($(i1)!{-2*\anglesize}!(i2)$)--($(i1)!{2*\anglesize}!(i2)$);
\path[name intersections={ of = angleconstruction and perpbisector, by={a1,a2}}];
\coordinate (anglepoint) at (a1);
\draw[myline] (as4)--(anglepoint)--(as1);
\end{scope}

\node [draw, myline] at (y1) [circle through={(p1)}] {};
\node [draw, myline] at (y2) [circle through={(p2)}] {};

\draw(C1) circle (\r);

\node at (xij) [below right=.0] {$x\ind{i,j}$};
\node at (xi1j) [below left=.0] {$x\ind{i+1,j}$};
\node at (xij1) [above right=.0] {$x\ind{i,j+1}$};
\node at (xi1j1) [above left] {$x\ind{i+1,j+1}$};
\node at (y1) [right] {$y^1\ind{i,j}$};
\node at (y2) [below] {$y^2\ind{i,j}$};

\foreach \p in {xij, xi1j, xi1j1, xij1, y1,y2}%, i1, anglepoint}
 	\path (\p) [mypoint=white];

\end{tikzpicture}

%% file: conicalnets.tex
\section{Conical nets}
\label{sec:conicalnets}

As in the previous discussion of multi-circular nets we start with describing properties of multi-conical nets in $\R^3$ first.

% \todo{
%   \begin{itemize}
  % \item Q-nets in Blaschke cylinder model of Laguerre geometry. 
  % \item What do the Laplace transforms mean?
%   \item \emph{What does the orthogonality of the Laplace transforms mean?}

%     Because of the degenerate bilinear form that defines the quadric, orthogonality is only captured by the non-degenerate part of the form.
%     This is polarity/orthogonality in the sense of 2-dimensional M\"obius geometry in $\Sphere^2$.
%     So the normals of the planes of a multi-conical net form a 2-dimensional multi-circular net and are generated by reflections/inversions in spheres defined by polar subspaces.
%   \item \emph{What do the reflections correspond to?}

%     Let $p = (1,n_1,n_2,n_3,d)$ be an oriented plane, i.e.\ a point on the Blaschke cylinder.
%     For the normals an L-reflection is an inversion in a sphere on~$\Sphere^2$.
%     This defines the new normal and the distance~$d'$ for the L-reflection in $(a_0, a_1, a_2, a_3, a_4)$ with
%     $\bilform{a}{a}_{3,1,0}=-a_0^2+a_1^2+a_2^2+a_3^2 = 1$ is
%     $d' = (d - 2 \bilform{p}{a} a_4)/(1-2\bilform{p}{a} a_0)$.
    
%   \item Consider embedded Gauss map only!
%   \end{itemize}
% }

\begin{definition}
  \label{def:conicalnet}
  A 2-dimensional \emph{discrete conical net} is a discrete $Q^*$-net
  $P: \Z^2 \to \{\text{planes in }\R^3\}$
  such that the four planes of an elementary quadrilateral are tangent to a common cone of revolution.
\end{definition}

The normals of the tangent planes of a cone of revolution lie on a circle.
Hence a $Q^*$-net~$P$ is conical if and only if the normals of the planes of~$P$, that is, the Gauss map of~$P$, is a circular net in~$\Sphere^2$.

\begin{definition}
  \label{def:multiconicalnet}
  A 2-dimensional \emph{multi-conical net} is a map
  $P: \Z^2 \to \{\text{planes in }\R^3\}$
  such the four planes
  $P\ind{i_0, j_0}, P\ind{i_1,j_0},  P\ind{i_1,j_1}, P\ind{i_0,j_1}$
  are tangent to a cone of revolution for all $i_0, i_1, j_0, j_1 \in \Z$.
\end{definition}

With the above observation on the circularity of the Gauss map of a conical net this yields a nice connection to multi-circular nets.

\begin{lemma}
  \label{lem:gauss-map-multiconical}
  Let $P: \Z^2 \to \{\text{planes in }\R^3\}$ be a 2-dimensional multi $Q^*$-net.
  Then~$P$ is a multi-conical net if and only if
  the Gauss map of~$P$ is a multi-circular net in~$\Sphere^2$.
\end{lemma}

This allows us to use a classification similar to the one of multi-circular nets of \Cref{thm:circular} to classify the Gauss maps of multi-conical nets.

\begin{theorem}[Gau{\ss} maps of multi-conical nets]
  \label{thm:multi-circular-s2}
  Let $x:\Z^2 \to \Sphere^2$ be a multi-circular net.
  Then it is a M\"obius transform of one of the following:
  \begin{itemize}
  \item % (+, ++-)
    a symmetric strip with respect to the equator,
  \item  % (-+,++)
    a net of revolution, or
  \item  % (0+,0+)
    the inverse stereographic projection of an orthogonal grid.
  \end{itemize}
\end{theorem}
\begin{proof}
  As in the proof of \Cref{thm:circular} we classify the nets by different decompositions of~$\R^{3,1}$ into polar subspaces in the projective model of M\"obius geometry.

  If one of the spaces has dimension one then we may normalize the corresponding circle to be the equator.
  Then the net is symmetric with respect to this equator and the odd resp.\ even vertices of the net coincide in the corresponding direction.

  Assume that both spaces of the polar decomposition have (a least) dimension two. Then we obtain the following cases:
  \begin{itemize}
  \item %(++,+-)
    The signatures of the polar subspaces~$U_1$ and $U_2$ are $({+}{+})$ and $({+}{-})$, respectively.
    The circles corresponding to the points in $U_1$ all pass through two points on the sphere.
    We can normalize this family such that the common points of the circles are antipodal points on~$\Sphere^2$.
    Then the circles in the corresponding pencil become great circles through the antipodal points and the circles of the other become horizontal parallel circles.
    Thus the net generated by the inversions in the circles of the two families is a rotational net.
  \item %(+0,+0)
    If the signatures of both of the polar subspaces is $({+}{0})$, then all circles pass through one point~$N$.
    Additionally, the circles of the families are touching a pair of orthogonal lines in this point.
    By a stereographic projection from~$N$ the two families of circles become families of orthogonal lines in the plane. 
    Hence the net is the image of an inverse stereographic projection of an orthogonal net in the plane. 
  \end{itemize}
\end{proof}

If we polarize a (multi-)Q net with respect to the sphere~$\Sphere^2$ we obtain a $Q^*$-net.
In particular, if we polarize a multi-circular net contained in the sphere we obtain a multi-conical net~$P$ with planes tangent to the sphere.
We will call a multi-conical net whose planes are tangent to the sphere a \emph{spherical multi-conical net}, see~\Cref{fig:multi-conical-construction} left.

By using the above classification of multi-circular nets in~$\Sphere^2$ we obtain the following classification of multi-conical nets.

\begin{theorem}
  \label{thm:conical}
  Let $P: \Z^2 \to \{\text{planes in } \R^3\}$ be a $Q^*$-net.
  Then $P$ is a multi-conical net if and only if it is parallel to a spherical multi-conical net.
  Furthermore, the spherical multi-conical net is the polar of its Gauss map.
\end{theorem}
\begin{proof}
  If $P$ is a multi-conical net, then its Gauss map is a multi-circular net in~$\Sphere^2$ and its polar is a spherical multi-conical net parallel to the original net.
  
  For the converse, let $S$ be a spherical multi-conical net.
  A multi-conical net is a special kind of multi-$Q^*$-net and hence all parameter polygons are planar.
  Now let $P$ be a parallel mesh of~$S$.
  Since the planes of $P$ are parallel to the corresponding planes of~$S$, the edges of~$P$ are also parallel to the edges of~$S$.
  But the parameter polygons of $S$ are planar.
  Hence the parameter polygons of~$P$ are also planar and~$P$ is a multi-$Q^*$-net.
  Additionally, the Gauss maps of~$P$ and~$S$ coincide.
  But as~$S$ is a multi-conical net, its Gauss map is a multi-circular net and hence the parallel net~$P$ (with the same Gauss map) is multi-conical.
  See~\Cref{fig:multi-conical-construction} for an illustration.
\end{proof}

\begin{figure}[tb]
  \includegraphics[width=.8\linewidth]{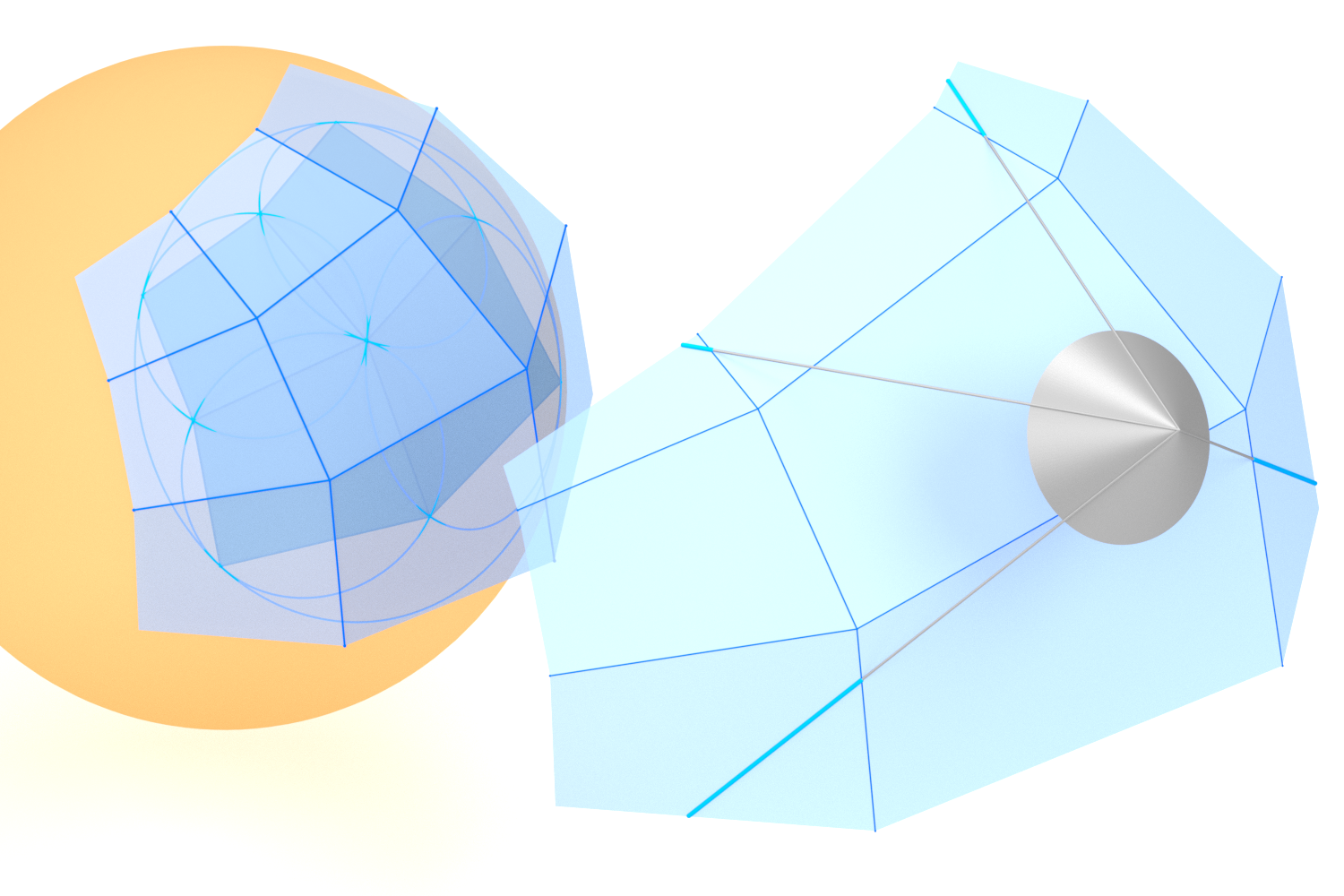}
  \caption{A multi-conical net (right) is a parallel net to a spherical multi-conical net (left). The polar net of a spherical multi-conical net is a multi-circular net (left). The four corner faces of the multi-conical nets are tangent to a common cone of revolution indicated by the four lines of contact and the cone (right) and the big circle of the spherical multi-circular net (left).}
  \label{fig:multi-conical-construction}
\end{figure}

\begin{remark}[Projective geometry perspective]
  Multi-conical nets can also be considered as multi-Q-nets in the Blaschke cylinder model~$\Sphere^2 \times \R$ of Laguerre geometry.
  Then they can be characterized and constructed using the general approach presented in \Cref{sec:trans_in_quadric}.
  In particular, we can use Laguerre reflections and one initial oriented plane to generate arbitrary multi-conical nets as described in \Cref{thm:reflections-translation-net}.
\end{remark}

To complete the picture, we characterize surface patches parametrized by a multi-conical net.
Consider a surface patch $x:[0,1]^2 \to \R^3$ with Gau{\ss} map $\hc n:[0,1]^2 \to \R^{3,1}$ in homogeneous coordinates, i.e., 
$\bilform{\hc n}{\hc n}_{3,1} = {\hc n_1}^2 + {\hc n_2}^2 + {\hc n_3}^2 - {\hc n_4}^2 = 0$.
Further let the tangent planes be given by the equations ${\hc n_1} x_1 + {\hc n_2} x_2 + {\hc n_3} x_3 = d$.
Combining the Gau{\ss} map~$\hc n$ and the distance~$d$, we obtain \emph{homogeneous tangent plane coordinates} $\hc t: [0,1]^2 \to \R^5$ with $\hc t = (\hc n, d) \in \R^5$.
(This is actually the description of the surface in terms of tangent planes in the Blaschke cylinder model of Laguerre geometry. Oriented planes correspond to points in the degenerate quadric associated to the bilinear form $\bilform{(\hc n, d)}{(\hc n, d)}_{3,1,1} = {\hc n_1}^2 + {\hc n_2}^2 + {\hc n_3}^2 - {\hc n_4}^2$ in $\R^{3,1,1}$.)

\begin{theorem}
  \label{thm:multi-conical-smooth}
  Let $x: [0,1]^2 \to \R^3$ be a parametrized surface patch such that the tangent planes of every parameter quadrilateral intersect in a point.
  Then the patch is parametrized along planar principal curvature lines.
  
  In particular, the Gau{\ss} map of the parameter lines consists of orthogonally intersecting circles and the homogeneous tangent plane coordinates form a projective translation surface in the quadric in~$\R^{3,1,1}$, i.e., $\hc t = (\hc p_1(u) + \hc p_2(v), d_1(u) + d_2(v))$.
\end{theorem}
\begin{proof}
  As in the proof of~\Cref{thm:multi-circular-smooth} consider discrete multi-conical nets sampled from the smooth surface.

  Multi-conical nets are in particular, multi-$Q^*$-nets. So if we take the limit of a multi-$Q^*$-net sampled on the surface, we observe, that the parameter lines are planar conjugate lines.
  Furthermore, sampling the Gau{\ss} map of the surface yields a multi-circular net in~$\Sphere^2$ by~\Cref{lem:gauss-map-multiconical}.  Hence the parameter lines of the Gau{\ss} map belong to orthogonally intersecting circles.
  But for each parameter line the plane containing the line and the plane containing the corresponding circle of the Gau{\ss} map are parallel.
  Hence the parameter lines are principal curvature lines.

  The special type of parametrization can for example be found in~\cite[\S 64]{Blaschke:DiffGeoIII}.
\end{proof}

A geometric description of the surfaces characterized in the above theorem is given in~\cite[\S 64]{Blaschke:DiffGeoIII}.
% In the generic case, the surfaces are characterized by their Gau{\ss} map and two functions~$U_1(u)$ and~$V_1(v)$:
% For a given Gau{\ss} map, there exists a unique Dupin cyclide up to parallel offset.
% For a Dupin cyclide there exists two planes tangent to two different curvature lines and a unique middle sphere, that is tangent to a third curvature line of the same family and the two planes.
% Now for each normal in the Gau{\ss} map there exist a tangent plane of the Dupin cyclide and a parallel tangent plane to the middle sphere. 
% Then there exists a parallel plane~$P$ such that the ratio of the distances of the three parallel planes is given by the functions~$U_1(u)$ and~$V_1(v)$.
% The planes~$P$ are the tangent planes of the general surface of~\Cref{thm:multi-conical-smooth}.
In the language of surface transformations these surfaces are Combescure transformations of Dupin cyclides.

%%% Local Variables:
%%% mode: latex
%%% TeX-master: "nets"
%%% End:

%% file: discrete-extensions.tex
\section{Discrete extensions}
\label{sec:discrete-extensions}

In this section we present a new approach to the extension of discrete nets.
Previously, we have always investigated the smooth extension of discrete nets, but here we would like to advocate a purely discrete point of view, that will lead to smooth extensions in the limit.
This extension of discrete nets by discrete multi-nets may also be interpreted 
as an interpolatory subdivision scheme. 

In contrast to previous subdivision schemes, we attach adapted multi-nets to the faces of a given net and obtain an
interpolatory subdivision scheme that preserves the structure of the underlying nets.
More precisely, the extension of Q-nets by multi-Q-nets will still be a Q-net and the extension of a circular net by a multi-circular net will produce ``finer'' circular nets.
In particular, we can refine our subdivision scheme arbitrarily and will obtain surfaces that interpolate the initial mesh.
The smoothness of the interpolating surface depends on the chosen subdivision.

The basis of the discrete extensions is a discrete analog of the smooth extension by projective translation surface patches described in \Cref{lem:unique-multi-Q-extension-4curves}.

\begin{proposition}
  \label{prop:unique-patch-4polygons}
  Let $Q \subset \RP^3$ be a planar quadrilateral with Laplace points~$y^1$ and~$y^2$ and $n \in \N$ with $n \ge 3$.
  Further let $p_0, p_1$ and $q_0, q_1$ be two pairs of polylines with $n$ vertices attached to the vertices of the quadrilateral as shown in \Cref{fig:discrete-Q-extension} such that $p_0$ and $p_1$ ($q_0$ and $q_1$) are in perspective with respect to $y^2$ (resp.\ $y^1$).
  Then there exists a unique adapted multi-Q-net $f: [0,\dots,n]^2 \to \RP^3$ such that
  \begin{itemize}
  \item
    $f(k,j) = p_j(k)$ for all $k \in [0,\dots,n]$ and $j \in \{0,1\}$, and
  \item
    $f(i,\ell) = q_i(\ell)$ for all $\ell \in [0,\dots,n]$ and $i \in \{0,1\}$.
  \end{itemize}
\end{proposition}

The proof can be performed in exactly the same way as the proof of the smooth analog~\Cref{lem:unique-multi-Q-extension-4curves} by normalizing the homogeneous coordinates of the polylines at the edges.

Then the discrete extension can be used to implement a subdivision scheme in the following way:

\begin{enumerate}
\item
  Attach polygons to all the edges of a Q-net such that polygons of opposite edges of a quadrilateral are in perspective with respect to the corresponding Laplace points.
\item
  Add a multi-Q-patch for each net quadrilateral to obtain a finer net.
\end{enumerate}

\Cref{fig:discrete-Q-extension} shows two steps of extension where we added two additional vertices along each edge.

\begin{figure}[bt]
  \includegraphics[width=.32\textwidth]{images/subdivision0.png}
  \includegraphics[width=.32\textwidth]{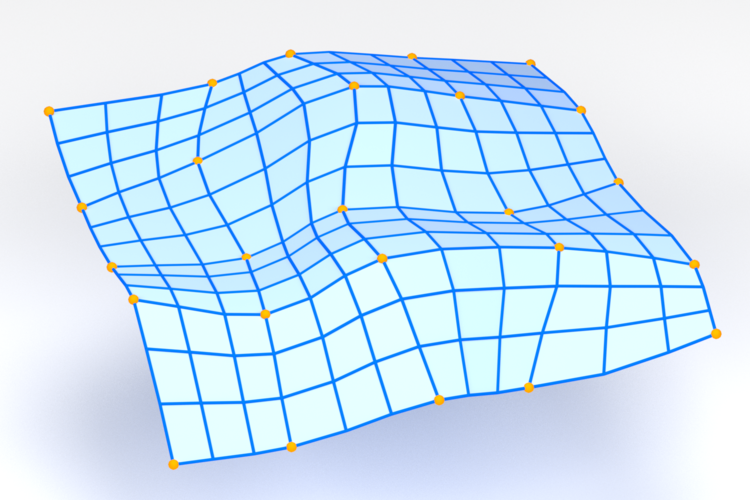}
  \includegraphics[width=.32\textwidth]{images/subdivision2.png}
  \caption{Planar quadrilateral subdivision/extension of a discrete Q-net}
  \label{fig:discrete-Q-extension}
\end{figure}

The following examples illustrate different uses of the degrees of freedom of the general subdivision.

\begin{example}
Add one additional vertex to each edge such that
\begin{itemize}
\item
  along the parameter lines the additional vertices for adjacent edges and the common vertex are collinear, and
\item
  the additional vertices of opposite edges are in perspective with respect to the corresponding Laplace points.
\end{itemize}
Then the multi-Q-nets attached to every quadrilateral are the control meshes of supercyclide surface patches. 
Hence using De Casteljau's Algorithm we obtain a supercyclidic net.
The $\mathcal{C}^1$-smoothness follows from the collinearity condition of the first and the last vertex of adjacent polylines along coordinate directions.
\end{example}

\begin{example}
We can also vary the number of vertices of the polylines attached to the edges of the base net to control the ``smoothness''.
If we subdivide the net with only a few vertices in one direction and many vertices in the other, then we obtain a net that consists of developable strips (see~\Cref{fig:conical-subd}).

\begin{figure}[t]
  \includegraphics[width=.4\textwidth]{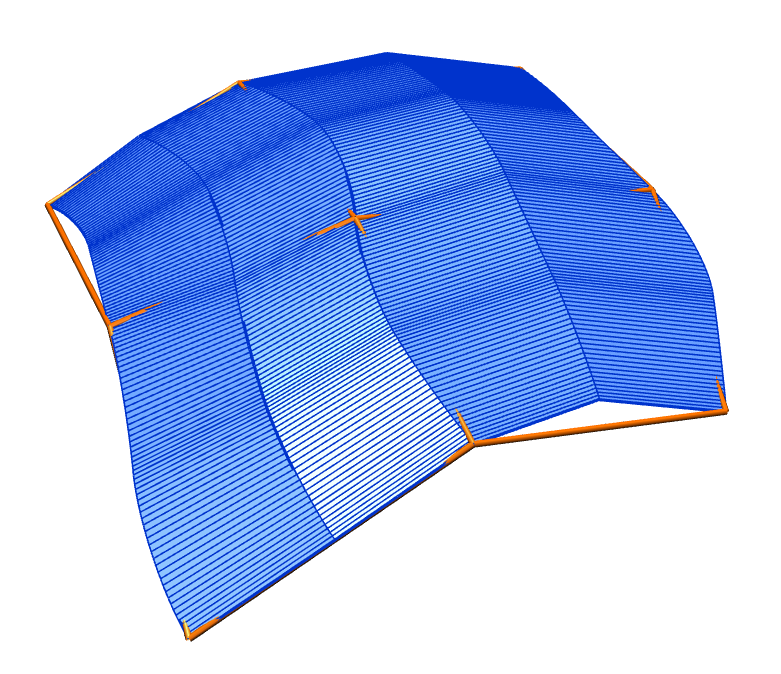}
  \includegraphics[width=.4\textwidth]{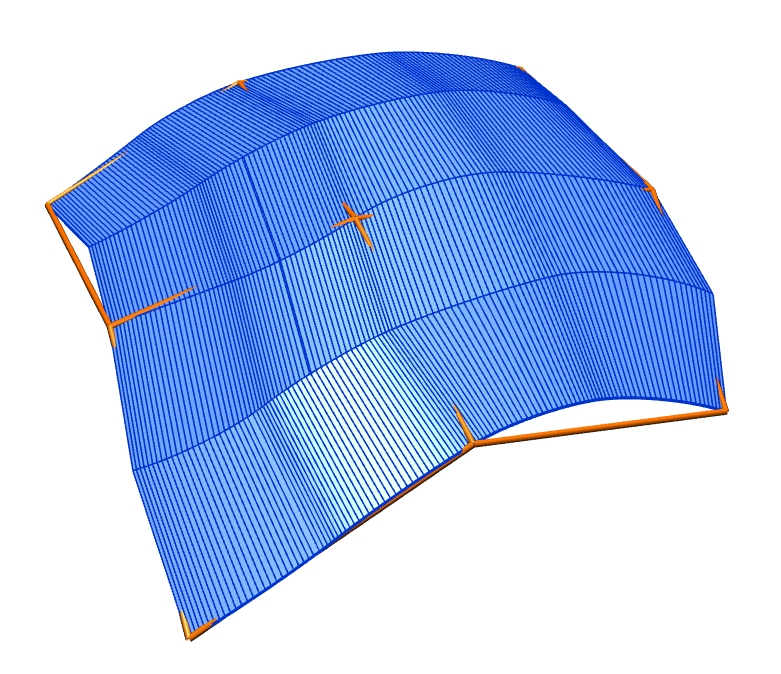}
  \caption{Choosing different vertex numbers for the two directions yields piecewise developable surface strips as a result of subdivision.}
  \label{fig:conical-subd}
\end{figure}
\end{example}

We leave further studies of the proposed structure-preserving subdivision schemes, which will include
shape control and multiresolution modeling techniques, for future research. 

%%% Local Variables: 
%%% mode: latex
%%% TeX-master: "nets"
%%% End: 

%% file: linecongruences.tex
\section{Multi line congruences}
\label{sec:isotropic_line_congruences}

In this section we follow the same approach for line congruences as for nets previously. The line congruences we are interested in are line congruences in quadrics, in particular, line congruences in the Lie and the Pl\"ucker quadrics. These isotropic line congruences represent principal contact element nets in Lie-sphere resp.\ Pl\"ucker line geometries. 

\begin{definition}
  \label{def:multi_line_congruence}
  A line congruence $l: \Z^2 \to \{\text{lines in $\RP^n$}\}$ is a \emph{multi-line congruence} if $\ell\ind{i_0,j_0}$ intersects $\ell\ind{i_1,j_0}$ and $\ell\ind{i_0,j_1}$ for all $i_0, i_1, j_0, j_1$.
\end{definition}

In other words, all lines associated with parameter lines intersect mutually. Hence, all lines $\ell_{i,j_0}$, $(i \in \Z)$ lie in a common plane, or pass through a common point. 
In Sections~\ref{sec:line-congruences-lie}, and \ref{sec:line-congruences-pluecker} we will use multi line congruences in the corresponding quadrics to characterize special principal contact element nets in Lie sphere and Pl\"ucker line geometries.

\subsection{Line congruences in the Lie quadric}
\label{sec:line-congruences-lie}

We will consider multi-line congruences in the Lie quadric~$\liequadric$ in the projective model of Lie sphere geometry in $\Proj(\R^{4,2})$. An isotropic line in the Lie quadric corresponds to a pencil of oriented spheres describing a contact element. Due to the signature of the Lie quadric of $(4,2)$ we are able to generate isotropic multi congruences in terms of points on the quadric, i.e., oriented spheres in $S^3$.

\begin{lemma}
  \label{lem:isotropic_lines_lie_quadric}
  Let $\ell:\Z^2 \to  \{\text{lines in } \liequadric\}$ be a generic isotropic multi line congruence. Then there exist points $s^1_i$ and $s^2_j$ in $\liequadric$ such that
  \[
  \ell\ind{i,j} = \lspan{s^1_i, s^2_j},
  \]
  for all $i,j \in \Z$. Furthermore, the subspaces $\lspan{s^1_i \,|\, i \in \Z}$ and $\lspan{s^2_j \,|\, j \in \Z}$ spanned by the two families of points are orthogonal.
\end{lemma}
\begin{proof}
  Consider three isotropic lines $\ell\ind{i_0,j_0}$, $\ell\ind{i_1,j_0}$, and $\ell\ind{i_2,j_0}$, associated with points on one parameter polygon. According to \Cref{def:multi_line_congruence}, the three lines intersect mutually and thus lie in a plane spanned by these lines or intersect in a common point. If $\ell\ind{i_0,j_0}$, $\ell\ind{i_1,j_0}$, and $\ell\ind{i_2,j_0}$ would lie in one plane, then this plane would have to be isotropic. But the Lie quadric does not contain isotropic planes. Hence the three lines intersect in one point denoted by $s^1_{i_0}$. In the generic case, the lines $\ell\ind{i_0,j}$ will span a subspace of dimension at least three, which implies that all the lines $\ell\ind{i_0,j}$ have to pass through the same point~$s^1_{i_0}$. 
  
  The same argument applies to the second direction, and hence we obtain the common point $s^2_{j_0}$ and the equation $\ell\ind{i,j} = \lspan{s^1_i, s^2_j}$ for all $i,j \in \Z$.

  Additionally, $s^1_{i_0}$ is the intersection of the isotropic lines $\ell\ind{i_0,j} = \lspan{s^1_{i_0}, s^2_j}$ and in particular $s^1_{i_0} \perp s^2_j$ for all $j \in \Z$.
\end{proof}
In terms of oriented spheres and contact elements, the above lemma states, that the contact elements $\ell\ind{i,j}$ are generated by the touching spheres  $s^1_i$ and $s^2_j$ for all $i,j \in \Z$.

Similar to the characterization of multi-Q-nets in the M\"obius quadric the characterization of isotropic multi congruences is based on the decomposition of a projective space into polar subspaces.

\begin{theorem}
  \label{thm:multi-congruence-lie}
  A generic multi line congruence in the Lie quadric consists of contact elements to a Dupin cyclide.
\end{theorem}
\begin{proof}
  By Lemma~\ref{lem:isotropic_lines_lie_quadric} the isotropic multi line congruence is generated by two orthogonal families of spheres $s^1_i, i \in \Z$ and $s^2_j, j \in \Z$ with $s^1_i \perp s^2_j$ for all $i,j \in \Z$.
  Let $S^1 = \lspan{ s^1_i}$ and $S^2 = \lspan{s^2_j}$. The only interesting case occurs, if both $S^1 \cap \liequadric$ and $S^2 \cap \liequadric$ contain more than two spheres. This happens if and only if the signature of the two polar subspaces is $({+}{+}{-})$ (all other cases can also easily be analyzed). But this characterizes exactly the contact elements of Dupin cyclides along curvature lines.
\end{proof}

\subsection{Line congruences in the Pl\"ucker quadric}
\label{sec:line-congruences-pluecker}

In this section we characterize multi line congruences in Pl\"ucker line geometry using the Pl\"ucker quadric~$\plueckerquadric$ in $\Proj(\R^{3,3})$. This case differs a little from the Lie geometric characterization since the Pl\"ucker quadric contains isotropic planes, namely $\alpha$- and $\beta$-planes.

\begin{lemma}
  \label{lem:isotropic_lines_pluecker_quadric}
  Let $\ell:\Z^2 \to \{\text{lines in } \plueckerquadric\}$ be a generic isotropic multi line congruence. Then there exist points $g^1_i$ and $g^2_j$ in $\plueckerquadric$ such that
  \[
  \ell\ind{i,j} = \lspan{g^1_i, g^2_j},
  \]
  for all $i,j \in \Z$. Furthermore, the subspaces $\lspan{g^1_i \,|\, i \in \Z}$ and $\lspan{g^2_j \,|\, j \in \Z}$ spanned by the two families of points are orthogonal.
\end{lemma}

\begin{proof}
  The proof is exactly the same as for Lemma~\ref{lem:isotropic_lines_lie_quadric}, except that we have to discuss the cases of planar families of lines: The Pl\"ucker quadric contains two kinds of isotropic planes, i.e., contact elements with a common plane or contact elements with a common point. In these cases the entire net will degenerate to a planar net, or a point and are hence excluded.
\end{proof}

As we have excluded the cases introduced by isotropic planes, the characterization of multi congruences in the Pl\"ucker quadric is exactly analogous to the Lie quadric case.

\begin{theorem}
  \label{thm:multi-congruence-pluecker}
  A generic multi line congruence in the Pl\"ucker quadric consists of contact elements to a hyperboloid (doubly ruled quadric).
\end{theorem}

%%% Local Variables:
%%% mode: latex
%%% TeX-master: "nets"
%%% End:

%% file: nets.bbl
\begin{thebibliography}{10}

\bibitem{Blaschke:DiffGeoIII}
W.~Blaschke.
\newblock {Vorlesungen \"uber Differentialgeometrie III}.
\newblock {X + 474 S. Berlin, J. Springer (Grundlehren der mathematischen
  Wissenschaften in Einzeldarstellungen) (1929).}, 1929.

\bibitem{BO20162}
P.~Bo, Y.~Liu, C.~Tu, C.~Zhang, and W.~Wang.
\newblock Surface fitting with cyclide splines.
\newblock {\em Computer Aided Geometric Design}, 43(Supplement C):2 -- 15,
  2016.
\newblock Geometric Modeling and Processing 2016.

\bibitem{Bo:2011:CAS}
P.~Bo, H.~Pottmann, M.~Kilian, W.~Wang, and J.~Wallner.
\newblock Circular arc structures.
\newblock {\em ACM Trans. Graph.}, 30(4):101:1--101:12, July 2011.

\bibitem{BobenkoHuhnenVenedey:Cyclides}
A.~I. {Bobenko} and E.~{Huhnen-Venedey}.
\newblock {C}urvature line parametrized surfaces and orthogonal coordinate
  systems: discretization with {D}upin cyclides.
\newblock {\em {Geom. Dedicata}}, 159:207--237, 2012.

\bibitem{BobHVRoer:Supercyclides}
A.~I. Bobenko, E.~Huhnen-Venedey, and T.~Rörig.
\newblock Supercyclidic nets.
\newblock {\em International Mathematics Research Notices}, 2017(2):323--371,
  2017.

\bibitem{BobenkoMatthesSuris:approximation}
A.~I. Bobenko, D.~Matthes, and Y.~B. Suris.
\newblock Discrete and smooth orthogonal systems: {$C^\infty$}-approximation.
\newblock {\em Int. Math. Res. Not.}, 2003(45):2415--2459, 2003.

\bibitem{BobenkoPottmannWallner:2010:CurvatureTheory}
A.~I. Bobenko, H.~Pottmann, and J.~Wallner.
\newblock A curvature theory for discrete surfaces based on mesh parallelity.
\newblock {\em Math. Annalen}, 348(1):1--24, 2010.

\bibitem{Bobenko:isothermic}
A.~I. Bobenko and Y.~B. Suris.
\newblock Isothermic surfaces in sphere geometries as {M}outard nets.
\newblock {\em Proceedings of the Royal Society of London A: Mathematical,
  Physical and Engineering Sciences}, 463(2088):3171--3193, 2007.

\bibitem{BobenkoSuris:DDG}
A.~I. {Bobenko} and Y.~B. {Suris}.
\newblock {\em {Discrete differential geometry. Integrable structure.}}
\newblock Providence, RI: American Mathematical Society (AMS), 2008.

\bibitem{Bol:ProjektiveDifferentialgeometrie}
G.~{Bol}.
\newblock {\em Projektive Differentialgeometrie I}, volume VII.
\newblock G\"ottingen: Vandenhoeck \& Ruprecht, 195.

\bibitem{Bouaziz:2012:ShapeUp}
S.~Bouaziz, M.~Deuss, Y.~Schwartzburg, T.~Weise, and M.~Pauly.
\newblock Shape-up: Shaping discrete geometry with projections.
\newblock {\em Comput. Graph. Forum}, 31(5):1657--1667, Aug. 2012.

\bibitem{Degen:1994:PlaneSilhouettes1}
W.~{Degen}.
\newblock {Nets with plane silhouettes.}
\newblock In {\em {Design and application of curves and surfaces. Mathematics
  of surfaces V. Based on the proceedings of the fifth conference on the
  mathematics of surfaces organized by the Institute of Mathematics and its
  Applications held in Edinburgh, UK, September 14-16, 1992}}, pages 117--133.
  Oxford: Clarendon Press, 1994.

\bibitem{Degen:1997:PlaneSilhouettes2}
W.~{Degen}.
\newblock {Nets with plane silhouettes. II.}
\newblock In {\em {The mathematics of surfaces. VII. Proceedings of the 7th
  conference, Dundee, Great Britain, September 1996}}, pages 263--279.
  Winchester: Information Geometers, Limited, 1997.

\bibitem{DoliwaSantini:1997:QnetsAreIntegrable}
A.~Doliwa and P.~M. Santini.
\newblock {Multidimensional quadrilateral lattices are integrable}.
\newblock {\em Phys. Lett. A}, 233:265--372, 1997.

\bibitem{Douthe:Isoradial}
C.~Douthe, R.~Mesnil, H.~Orts, and O.~Baverel.
\newblock Isoradial meshes: Covering elastic gridshells with planar facets.
\newblock {\em Automation in Construction}, 83(Supplement C):222 -- 236, 2017.

\bibitem{Dynetal-1990-butterfly}
N.~Dyn, D.~Levin, and J.~Gregory.
\newblock A butterfly subdivision scheme for surface interpolation with tension
  control.
\newblock {\em ACM Transactions on Graphics}, 9(2):160--169, 1990.

\bibitem{grohs-2008-subd}
P.~Grohs.
\newblock Smoothness analysis of subdivision schemes on regular grids by
  proximity.
\newblock {\em SIAM J. Numerical Analysis}, 46:2169--2182, 2008.

\bibitem{HertrichJeromin:MoebiusDiffGeo}
U.~{Hertrich-Jeromin}.
\newblock {\em {Introduction to M\"obius differential geometry.}}
\newblock Cambridge: Cambridge University Press, 2003.

\bibitem{HVRoer:HyperbolicNets}
E.~Huhnen-Venedey and T.~R{\"o}rig.
\newblock Discretization of asymptotic line parametrizations using hyperboloid
  surface patches.
\newblock {\em Geom. Dedicata}, 168(1):265--289, 2014.

\bibitem{Kobbelt-1996-intsubd}
L.~Kobbelt.
\newblock Interpolatory subdivision on open quadrilateral nets with arbitrary
  topology.
\newblock {\em Comput. Graph. Forum}, 15(3):409--420, 1996.

\bibitem{liu-2006-cm}
Y.~Liu, H.~Pottmann, J.~Wallner, Y.-L. Yang, and W.~Wang.
\newblock Geometric modeling with conical meshes and developable surfaces.
\newblock {\em ACM Trans. Graphics}, 25(3):681--689, 2006.
\newblock Proc. SIGGRAPH.

\bibitem{Mesnil:2016:Marionette}
R.~Mesnil, C.~Douthe, O.~Baverel, and B.~Léger.
\newblock Marionette mesh: from descriptive geometry to fabrication-aware
  design.
\newblock In S.~Adriaenssens et~al., editors, {\em Advances in Architectural
  Geometry 2016}, pages 62--81, Switzerland, 2016. Hochschulverlag Z{\"u}rich.

\bibitem{Sauer:1937:ProjLinienGeometrie}
R.~Sauer.
\newblock {\em {Projektive Liniengeometrie}}.
\newblock {Berlin, W. de Gruyter \& Co. (G\"oschens Lehrb\"ucherei I. Gruppe,
  Bd. 23)}, 1937.

\bibitem{tang-2014-ff}
C.~Tang, X.~Sun, A.~Gomes, J.~Wallner, and H.~Pottmann.
\newblock Form-finding with polyhedral meshes made simple.
\newblock {\em ACM Trans. Graphics}, 33(4), 2014.
\newblock {P}roc. SIGGRAPPH.

\bibitem{vaxmanetal-2018-moebiussubd}
A.~Vaxman, C.~M{\"u}ller, and O.~Weber.
\newblock M{\"o}bius invariant mesh subdivision, 2018.
\newblock private communication.

\bibitem{net:10.2312:CGF.v29i5pp1671-1679}
M.~Zadravec, A.~Schiftner, and J.~Wallner.
\newblock {Designing Quad-dominant Meshes with Planar Faces}.
\newblock {\em Computer Graphics Forum}, 2010.

\bibitem{zorinetal-1996-intsubd}
D.~Zorin, P.~Schr{\"o}der, and W.~Sweldens.
\newblock Interpolating subdivision for meshes with arbitrary topology.
\newblock In {\em Proc. SIGGRAPH '96}, pages 189--192, New York, 1996. ACM.

\end{thebibliography}
